\documentclass[11pt]{article}
\usepackage{fullpage}

\usepackage[margin=10pt,font=small,labelfont=bf,labelsep=colon]{caption}
\usepackage[pagebackref=true,hidelinks]{hyperref}
\renewcommand*{\backrefalt}[4]{%
\ifcase #1 \footnotesize{(Not cited.)}%
\or        \footnotesize{(Cited on p.~#2)}%
\else      \footnotesize{(Cited on pp.~#2)}%
\fi}

\usepackage[bottom]{footmisc}

\def\MODE{3}

\usepackage{algorithmic}
\usepackage{textcomp}
\usepackage{graphicx}
\graphicspath{{./graphics/}}
\usepackage{color}

\usepackage{cite}                   
\usepackage[utf8]{inputenc}			
\usepackage[english]{babel}			
\usepackage{enumerate}
\usepackage{bbm}                    

\usepackage{math}                   

\if\MODE3
	\def\qed{\rule[0pt]{5pt}{5pt}\par\medskip}
	\newcommand{\qedhere}{\hfill ~\qed}
	\newenvironment{proof}{{\noindent\bf Proof.}}{\qedhere}
\else
	\usepackage{generic}
\fi

\usepackage{ntheorem}
\usepackage[capitalise,nameinlink]{cleveref}

\newtheorem{thm}{Theorem}

\newtheorem{lem}[thm]{Lemma}

\newtheorem{rem}[thm]{Remark}
\newtheorem{defn}[thm]{Definition}

\newtheorem{assumption}[thm]{Assumption}

\crefname{thm}{Theorem}{Theorems}
\crefname{problem}{Problem}{Problems}
\crefname{lem}{Lemma}{Lemmas}
\crefname{prop}{Proposition}{Propositions}
\crefname{rem}{Remark}{Remarks}
\crefname{defn}{Definition}{Definitions}
\crefname{cor}{Corollary}{Corollaries}
\crefname{assumption}{Assumption}{Assumptions}

\if\MODE3
	\usepackage{parskip}

	\makeatletter
	\def\thm@space@setup{%
	\thm@preskip=\parskip \thm@postskip=0pt
	}
	\makeatother
\fi

\newcommand{\1}{\mathbf{1}}
\newcommand{\0}{\mathbf{0}} 

\newcommand{\Pp}{\mathcal{P}}
\newcommand{\Pu}{\mathcal{P}_{11}}
\newcommand{\Pv}{\mathcal{P}_{12}}
\newcommand{\K}{\mathcal{K}}
\newcommand{\Pw}{\mathcal{P}_{21}}	
\newcommand{\Pg}{\mathcal{P}_{22}}
\newcommand{\Q}{\mathcal{Q}}
\newcommand{\Ss}{\mathcal{S}}
\newcommand{\Stau}{\mathcal{S}_{\tau}}

\newcommand{\Tx}{\mathcal{T}}
\newcommand{\F}{\mathcal{F}}
\newcommand{\G}{\mathcal{G}}

\newcommand{\proper}{\mathcal{L}_\textup{prop}}

\newcommand{\Dcen}{J_{\textup{cen}}}
\newcommand{\Ddecen}{J_{\textup{dec}}}

\newcommand{\Ddecdel}{J_{\textup{dec},\textup{del}}}
\newcommand{\Dcendel}{J_{\textup{del}}}
\newcommand{\DQ}{J_Q}

\newcommand{\Qopt}{\Q_{\textup{opt}}}

\newcommand{\sbmat}[1]{\left(\begin{smallmatrix}#1\end{smallmatrix}\right)}
\newcommand{\squeezemat}[1]{\addtolength{\arraycolsep}{-#1}}

\newcommand{\blue}[1]{\textcolor{blue}{#1}}
\renewcommand{\blue}[1]{#1}

\if\MODE3\else
	\def\BibTeX{{\rm B\kern-.05em{\sc i\kern-.025em b}\kern-.08em
			T\kern-.1667em\lower.7ex\hbox{E}\kern-.125emX}}
\fi

\begin{document}

\if\MODE3 
	\title{Optimal Control of Multi-Agent Systems\\ with Processing Delays}
	\date{}
	\author{
	Mruganka Kashyap\thanks{M.~Kashyap is with the Department of Electrical and Computer Engineering at\\ Northeastern University, Boston, MA 02115, USA. (e-mail: kashyap.mru@northeastern.edu).}
	\and 
	Laurent Lessard\thanks{L.~Lessard is with the Department of Mechanical and Industrial Engineering at\\
	Northeastern University, Boston, MA 02115, USA. (e-mail: l.lessard@northeastern.edu).}}
\else
	\title{Optimal Control of Multi-Agent Systems with Processing Delays}
	\author{Mruganka Kashyap, \IEEEmembership{Member, IEEE}, and Laurent Lessard, \IEEEmembership{Senior Member, IEEE}
	\thanks{This paragraph of the first footnote will contain the date on 
			which you submitted your paper for review. It will also contain support 
			information, including sponsor and financial support acknowledgment. \\	
			This material is based upon work supported by the National Science Foundation under Grant No.~2136317.}
	\thanks{M.~Kashyap completed this work when he was with the Department of Electrical and Computer Engineering at Northeastern University, Boston, MA 02115, USA. He is now at Indigo Technologies, Inc.\\(e-mail: mruganka.kashyap@indigotech.com).}
	\thanks{L.~Lessard is with the Department of Mechanical and Industrial Engineering at Northeastern University, Boston, MA 02115, USA.\\(e-mail: l.lessard@northeastern.edu).}}
\fi
\maketitle


\begin{abstract}
    In this article, we consider a cooperative control problem involving \blue{a heterogeneous network of dynamically decoupled continuous-time linear plants. The (output-feedback) controllers for each plant may communicate with each other according to a fixed and known transitively closed directed graph. Each transmission incurs a fixed and known time delay.} We provide an explicit closed-form expression for the optimal decentralized controller and its associated cost under these communication constraints and standard linear quadratic Gaussian (LQG) assumptions for the plants and cost function. We find the exact solution without discretizing or otherwise approximating the delays. We also present an implementation of each sub-controller that is efficiently computable, and is composed of standard finite-dimensional linear time-invariant (LTI) and finite impulse response (FIR) components, and has an intuitive observer-regulator architecture reminiscent of the classical separation principle.
\end{abstract}


\section{Introduction}\label{sec:intro}

In multi-agent systems such as swarms of unmanned aerial vehicles, it may be desirable for agents to cooperate in a decentralized fashion without receiving instructions from a central coordinating entity. Each agent takes local measurements, performs computations, and may communicate its measurements with a given subset of the other agents, with a time delay. In this \blue{article}, we investigate the problem of optimal control under the aforementioned communication constraints.

We model each agent as a continuous-time \blue{linear time-invariant (LTI)} system. \blue{We make no assumption of homogeneity across agents; each agent may have different dynamics.} We assume the aggregate dynamics of all agents are described by the state-space equations
\begin{equation}\label{aggregateddynamics}
\bmat{\dot x \\ z \\ y} = \bmat{
    A & B_1 & B_2 \\ 
    C_1 & 0 & D_{12} \\
    C_2 & D_{21} & 0}
\bmat{x \\ w \\ u},
\end{equation}
where $x$ is the global state, $z$ is the regulated output, $y$ is the measured output, $w$ is the exogenous disturbance, and $u$ is the controlled input. The decoupled nature of the agents imposes a sparsity structure on the plant. Namely, if we partition $x$, $y$, $w$, $u$ each into $N$ pieces corresponding to the $N$ agents, the conformally partitioned state space matrices $A$, $B_1$, $B_2$, $C_2$, $D_{21}$ are block-diagonal. The regulated output $z$, however, couples all agents' states and inputs, so in general $C_1$ and $D_{12}$ will be dense.
The matrix transfer function $(w,u)\to (z,y)$ is a standard four-block plant that takes the form\footnote{In a slight abuse of notation, the vectors $z$, $y$, $w$, and $u$ now refer to the Laplace transforms of the corresponding time-domain signals in \eqref{aggregateddynamics}.}
\begin{equation}\label{fourblockplant}
	\bmat{z \\ y} = \bmat{\Pp_{11}(s) & \Pp_{12}(s) \\ \Pp_{21}(s) & \Pp_{22}(s)}
	\bmat{ w\\ u},
\end{equation}
where $\Pp_{21}$ and $\Pp_{22}$ are block-diagonal.

\blue{We assume information sharing is mediated by a fixed and known directed graph. Specifically, if there is a (possibly multi-hop) directed path from Agent $i$ to Agent $j$, then Agent $j$ can observe the local measurements of Agent $i$ with a delay $\tau$. We further assume there are no self-delays, so agents can observe their local measurements instantaneously.}

\blue{In practice, our setting corresponds to a network where the chief source of latency is due to processing and transmission delays~\cite[\S1.4]{kurose2021networking} (the encoding, decoding, and transmission of information). Therefore, we neglect propagation delays (proportional to distance traveled) and queuing delays (related to network traffic and hops required to reach the destination).}

We assume $\tau$ is fixed and known \blue{and homogeneous across all communication paths,} as it is determined by the physical capabilities (e.g., underlying hardware and software) of the individual agents rather than external factors. Thus, Agent~$i$'s feedback policy (in the Laplace domain) is of the form\footnote{There is no loss of generality in assuming a linear control policy; see \cref{sec:lit_review} for details.}
\begin{equation}\label{Kform}
u_i = \K_{ii}(s)y_i + \sum_{j \to i} e^{-s\tau}\K_{ij}(s) y_j,
\end{equation}
where the sum is over all agents $j$ for which there is a directed path from $j$ to $i$ in the underlying communication graph.

Given the four-block plant \eqref{fourblockplant}, the directed communication graph, and the processing delay $\tau$, we study the problem of finding a structured controller that is internally stabilizing and minimizes the $\Htwo$ norm of the closed-loop map $w\to z$.

In spite of the non-classical information structure present in this problem, it is known that there is a convex Youla-like parameterization of the set of stabilizing structured controllers, and the associated $\Htwo$ synthesis problem is a convex, albeit infinite-dimensional, optimization problem.

\textbf{Main contribution.} We provide a complete solution to this structured cooperative control problem that is computationally tractable and intuitively understandable. Specifically, the optimal controller can be implemented with a finite memory and transmission bandwidth that does not grow over time. Moreover, the controller implementations at the level of individual agents have \blue{separation structures between the observer and regulator} reminiscent of classical $\Htwo$ synthesis theory.\looseness=-1

In the remainder of this section, we give context to this problem and relate it to works in optimal control, delayed control, and decentralized control.
In \cref{sec:notation}, we cover some mathematical preliminaries and give a formal statement of the problem.
In \cref{sec:result_consolidated}, we give a convex parameterization of all structured suboptimal controllers, and present the $\Htwo$-optimal controller for the non-delayed ($\tau=0$) and delayed ($\tau > 0$) cases. In \cref{sec:result_3}, we describe the optimal controller architecture at the level of the individual agents, and give intuitive interpretations of the controller architecture.
In \cref{sec:result_4}, we present case studies that highlight the trade-offs between processing delay, connectivity of the agents, and optimal control cost. Finally, we conclude in \cref{sec:conclufuture} and discuss future directions.

\subsection{Literature review}\label{sec:lit_review}

If we remove the structural constraint~\eqref{Kform} and allow each $u_i$ to have an arbitrary causal dependence on all $y_j$ with no delays, the optimal controller is linear and admits an observer--\blue{regulator} separation structure \cite{wonham1968separation}. This is the classical $\Htwo$ (LQG) synthesis problem, solved for example in \cite{ZDG}.

The presence of structural constraints generally leads to an intractable problem~\cite{blondel}. For example, linear compensators can be strictly suboptimal, even under LQG assumptions~\cite{witsenhausen1968counterexample}. Moreover, finding the best \emph{linear} compensator also leads to a non-convex infinite-dimensional optimization problem.\looseness=-1

However, not all structural constraints lead to intractable synthesis problems. For LQG problems with \blue{\emph{partially nested} information}, there is a linear optimal controller~\cite{ho1972team}. If the information constraint is \blue{\emph{quadratically invariant}} with respect to the plant, the problem of finding the optimal LTI controller can be convexified~\cite{rotkowitz2005characterization,rotkowitz2010convexity}. The problem considered in this \blue{article} is both \blue{partially nested and quadratically invariant}, so there is no loss in assuming a linear policy as we do in~\eqref{Kform}.

Once the problem is convexified, the optimal controller can be computed exactly using approaches like vectorization \cite{rotkowitz_vectorization,vamsi_elia}, or approximated to arbitrary accuracy using Galerkin-style numerical approaches~\cite{voulgaris_stabilization,scherer02}.
However, these approaches lead to realizations of the solution that are neither minimal nor easily interpreted. For example, a numerical solution will not reveal a separation structure in the optimal controller, nor will it provide an interpretation of controller states or the signals communicated between agents' controllers.
Indeed, the optimal controller may have a rich structure, reminiscent of the centralized separation principle. Such \emph{explicit solutions} were found for broadcast~\cite{lessard2012decentralized}, triangular~\cite{tanaka2014optimal, lessard2015optimal}, and dynamically decoupled~\cite{kim2012optimal,kim2015explicit,kashyap2019explicit} cases.

The previously mentioned works do not consider time delays. In the presence of delays, we distinguish \blue{between discrete and continuous time. In discrete time, the delay transfer function $z^{-1}$ is rational}. Therefore, the problem may be reduced to the non-delayed case by absorbing each delay into the plant~\cite{spdel}. However, this reduction is not possible in continuous time because the \blue{continuous-time delay transfer function $e^{-s\tau}$ is irrational}. A Pad\'{e} approximation may be used for the delays~\cite{yan1996teleoperation}, but this leads to approximation error and a larger state dimension.

Although the inclusion of continuous-time delays renders the state space representation infinite-dimensional, the optimal controller may still have a rich structure. For systems with a \emph{dead-time delay} (the entire control loop is subject to the same delay), a loop-shifting approach using \blue{finite impulse response (FIR)} blocks can transform the problem into an equivalent delay-free LQG problem with a finite-dimensional LTI plant~\cite{mirkin2003every,mirkin2003extraction}. A similar idea was used in the discrete-time case to decompose the structure into dead-time and FIR components, which can be optimized separately~\cite{lampHtwoDelay}.

The loop-shifting technique can be extended to the \emph{adobe delay} case, where the feedback path contains both a delayed and a non-delayed path \cite{mirkin2009loop,mirkin2011dead,mirkin2012h2}.
The loop-shifting technique was also extended to specific cases like bilateral teleoperation problems that involve two stable plants whose controllers communicate across a delayed channel~\cite{kristalny2012decentralized,cho2012h2}, and haptic interfaces that have two-way communication with a shared virtual environment~\cite{kristalny2013decentralized}. \blue{Another example is the case of homogeneous agents coupled via a diagonal-plus-low-rank cost~\cite{madjidian2016h2}. All three of} these examples are special cases of the information structure~\eqref{Kform}. 

In the present work, we solve a general structured $\Htwo$ synthesis problem with $N$ agents that communicate using a structure of the form~\eqref{Kform}. We present explicit solutions that show an intuitive \blue{observer-regulator} structure at the level of each individual sub-controller. 
Preliminary versions of these results that only considered stable or non-delayed plants were reported in~\cite{kashyap2019explicit,kashyap2020agent}. \blue{In this article, we consider the general case of an unstable plant, we find an agent-level parameterization of all stabilizing controllers, and we obtain explicit closed-form expressions for the optimal cost.}

\section{Preliminaries}\label{sec:notation}

\paragraph{Transfer matrices.}

\blue{Let $\C_\alpha \defeq \set{s \in \C}{\Re(s) > \alpha}$ and $\bar{\C}_\alpha \defeq \set{s \in \C}{\Re(s) \geq \alpha}$.
A transfer matrix $\G(s)$ is said to be \emph{proper} if there exists an $\alpha > 0$ such that $\sup_{s \in \C_\alpha} \norm{\G(s)} <\infty$. We call this set $\proper$. Similarly, a transfer matrix $\G(s)$ is said to be \emph{strictly proper} if this supremum vanishes as $\alpha \rightarrow \infty$.
The Hilbert space $\Ltwo$ consists of analytic functions $\F:i\R\to \C^{m\times n}$ equipped with the inner product 
$\langle\F,\G\rangle\defeq \frac{1}{2 \pi}\int_\R \trace\bigl(\F(i\omega)^*\G(i\omega)\bigr) \,\mathrm{d}\omega
$, where the inner product induced norm $\norm{\F}_2\defeq {\langle\F,\F\rangle}^{1/2}$ is bounded. A function $\F:\bar{\C}_{0}\to \C^{m\times n}$ is in $\Htwo$ if $\F(s)$ is analytic in $\C_0$, $\textup{lim}_{\sigma\rightarrow 0^{+}} \F\left(\sigma+i\omega\right)=\F\left(i\omega\right)$ for almost every $\omega\in\R$, and $\sup_{\sigma\geq 0} \frac{1}{2 \pi} \int_{-\infty}^{\infty} \trace\bigl(\F(\sigma+i\omega)^*\F(\sigma+i\omega)\bigr) \,\mathrm{d}\omega<\infty$. This supremum is always achieved at $\sigma=0$ when $\F\in\Htwo$. The set $\Htwo^\perp$ is the orthogonal complement of $\Htwo$ in $\Ltwo$.
The set $\RHtwo$ refers to the subspace of strictly proper rational transfer functions with no poles in $\bar{\C}_0$. Similarly, the set $\RHtwo^\perp$ refers to the subspace of strictly proper rational transfer functions with all poles in $\C_0$. The set $\Linf$ consists of matrix-valued functions $\F:i\R\to \C^{m\times n}$ for which $\sup_{\omega\in\R}\norm{\F(i\omega)} < \infty$. $\Hinf$ and $\RHinf$ are defined analogously to $\Htwo$ and $\RHtwo$.}

The state-space notation for transfer functions is
\begin{align}\label{ssnotation}
	\G(s) = \left[\begin{array}{c|c}%
		A&B\\ \hlinet C&D\end{array}\right] \defeq D+C(sI-A)^{-1}B.
\end{align}
A square matrix $A$ is \emph{Hurwitz} \blue{if none of its eigenvalues belong to $\C_0$.} If $A$ is Hurwitz in \eqref{ssnotation}, then \blue{$\G\in\RHinf$}. If $A$ is Hurwitz and $D=0$, then \blue{$\G\in\RHtwo$}.
\blue{The \emph{conjugate} of $\G$ is
\begin{align*}
	\G^\sim(s) = \G^\tp(-s)=\left[\begin{array}{c|c}%
		-A^\tp&C^\tp\\\hlinet {-B^\tp}&D^\tp\end{array}\right].
\end{align*}} 
The dynamics \eqref{aggregateddynamics} and four-block plant $\Pp$ from \eqref{fourblockplant} satisfy
\begin{equation}\label{eq:fourblock}
\Pp(s) \defeq
\bmat{ \Pu(s) & \Pv(s) \\ \Pw(s) & \Pg(s) } =
\left[\begin{array}{c|cc}
	A & B_1 & B_2\\ \hlinet
	C_1 & 0 & D_{12} \\
	C_2 & D_{21} & 0 \end{array}\right].
\end{equation}
If we use the feedback policy $u = \K y$, then we can eliminate $u$ and $y$ from \eqref{fourblockplant} to obtain the closed-loop map $w\to z$, which is given by the lower linear fractional transformation (LFT) defined as $\F_l(\Pp,\K)\defeq \Pu+\Pv\K(I-\Pg\K)^{-1}\Pw$. LFTs can be inverted: if $\K=\F_l(\J,\Q)$ and $\J$ has a proper inverse, then $\Q=\F_u(\J^{-1},\K)$, where $\F_u$ is the upper linear fractional transformation: $\F_u(\Pp,\K)\defeq \Pg+\Pw\K(I-\Pu\K)^{-1}\Pv$.

\paragraph{Block indexing.}
Ordered lists of indices are denoted using $\{\ldots\}$.
The total number of agents is $N$ and $[N] \defeq \{1,\dots,N\}$. The $i^\text{th}$ subsystem has state dimension $n_i$, input dimension $m_i$, and measurement dimension $p_i$. The global state dimension is $n \defeq n_1+\cdots+n_N$ and similarly for $m$ and $p$. The matrix $I_k$ is the identity of size $k$ and $\blkdiag(\{X_i\})$ is the block-diagonal matrix formed by the blocks $\{X_1,\dots,X_n\}$.
The zeros used throughout are matrix or vector zeros and their sizes are dependent on the context.

We write $\underline{i}$ to denote the \textit{descendants} of node $i$, i.e., the set of nodes $j$ such that there is a directed path from $i$ to $j$ for all $i\in [N]$. By convention, we list $i$ first, and then the remaining indices \blue{in increasing order}. The directed path represents the direction of information transfer between the agents. Similarly, $\bar{i}$ denotes the \textit{ancestors} of node $i$ (again listing $i$ first). We also use $\bar{\bar{i}}$ and $\underline{\underline{i}}$ to denote the \textit{strict} ancestors and descendants, respectively, which excludes $i$.
For example, in \cref{fig:cartoon}, we have \blue{$\underline{2} = \{2,5\}$ and $\bar{\bar{3}} = \{1,4\}$.}

\begin{figure}[ht]
	\centering
	\includegraphics{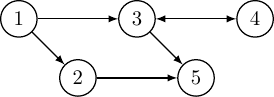}
	\caption{Directed graph representing \blue{five} interconnected systems.}
	\label{fig:cartoon}
\end{figure}

We also use this notation to index matrices. For example, if $X$ is a \blue{$5\times 5$} block matrix, then
\blue{$X_{1\underline{2}} = \bmat{X_{12} & X_{15}}$.} 
We will use specific partitions of the identity matrix throughout: $I_{n} \defeq \blkdiag(\{I_{n_i}\})$, and for each agent $i \in [N]$, we define $E_{n_i} \defeq (I_{n})_{:i}$ (the $i^\textup{th}$ block column of $I_{n}$). We have $n_{\underline{i}}=\sum_{k\in \underline{i}}n_k$ and $n_{\bar{i}}=\sum_{k\in \bar{i}}n_k$, akin to the descendant and ancestor definitions above. The dimensions of $E_{{n_{\bar{i}}}}$ and $E_{{n_{\underline{i}}}}$ are determined by the context of use. \blue{We also use the notations $X_{:i}$ and $X_{\bar{i}:}$ to indicate the $i^\textup{th}$ block column and $\bar{i}^\textup{th}$ block rows respectively for a matrix $X$.} Similar notations $1_{n}$ is the $n\times 1$ matrix of $1$'s. Further  notations are defined at their points of first use.

\subsection{Delay} \label{sec:delay}

We follow the notation conventions set in~\cite{mirkin2012h2}. 
The \textit{adobe delay} matrix \if\MODE3\else{defined as}\fi \blue{\;$\Lambda_m^i \defeq \blkdiag(I_{m_i},e^{-s\tau}I_{m_{\underline{\underline{i}}}})$} leaves block $i$ unchanged and imposes a delay of $\tau$ on all strict descendants of $i$.
We define $\Gamma : (\Pp,\Lambda_m^i) \mapsto (\tilde{\Pp},\Pi_u,\Pi_b)$ that maps the plant $\Pp$ in~\eqref{eq:fourblock} and adobe delay matrix $\Lambda_m^i$ to a modified plant $\tilde{\Pp}$ and \blue{FIR} systems $\Pi_u$ and $\Pi_b$. This loop-shifting transformation reported in \cite{mirkin2009loop,mirkin2012h2,mirkin2011dead} shown in \cref{fig:Mirkin_Gamma} transforms a loop with adobe input delay into a modified system involving a rational plant $\tilde\Pp$.
See \cref{sec:appendix_gamma} for details on the definition of $\Gamma$.

In this decomposition, $\langle\Delta,\Psi\rangle = 0$ and $\Psi$ is inner \blue{(if $\Psi \in \RHinf$ and $\Psi^{\sim}\Psi = I$)}, so the closed-loop map satisfies $\norm{ \F_l( \Pp,\Lambda_m^i\K ) }^2 = \norm{\Delta}^2 + \norm{\F_l( \tilde\Pp,\tilde\K)}^2$.
Thus, we can find the $\Htwo$-optimal $\K$ by first solving a standard $\Htwo$ problem with $\tilde\Pp$ to obtain $\tilde\K$, and then transforming back using $\K = \Pi_u \tilde\K (I - \Pi_b\tilde\K)^{-1}$. This transformation, illustrated in the bottom left panel of \cref{fig:Mirkin_Gamma}, has the form of a \emph{modified Smith predictor}, where the FIR blocks $\Pi_u$ and $\Pi_b$ compensate for the effect of the adobe delay in the original loop. See \cite[\S III.C]{mirkin2011dead} for further detail.

\begin{figure}[ht]
	\centering
	\includegraphics{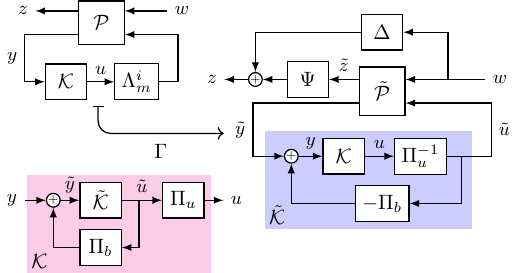}
	\caption{The loop-shifting approach \cite{mirkin2009loop,mirkin2012h2,mirkin2011dead} transforms a loop with adobe input delay (top left) into a modified system involving a rational plant $\tilde\Pp$ and FIR blocks $\Pi_u$ and $\Pi_b$ (right). This transformation $\Gamma$ is defined in \cref{sec:appendix_gamma}. We can recover $\K$ from $\tilde \K$ via the inverse transformation (bottom left).}
	\label{fig:Mirkin_Gamma}
\end{figure}

\subsection{Problem statement}\label{sec:problem_statement}
Consider a four-block plant~\eqref{eq:fourblock} representing the aggregated dynamics of $N$ agents as described in \cref{sec:intro}, which we label using indices $i \in [N]$. Suppose $x \in \R^n$, $u \in \R^m$, and $y\in \R^p$, partitioned conformally with the $N$ subsystems as $n = n_1 + \cdots + n_N$ and similarly for $m$ and $p$.

Consider a directed graph on the nodes $[N]$, and let $\Stau$ be the set of compensators of the form~\eqref{Kform}. For example, for the directed graph of \cref{fig:cartoon}, \blue{every controller takes the form}
\blue{\[
\bmat{\K_{11} & 0 & 0 & 0 & 0 \\
	e^{-s\tau}\K_{21} & \K_{22} & 0 & 0 & 0 \\
	e^{-s\tau}\K_{31} & 0 & \K_{33} & e^{-s\tau}\K_{34} & 0 \\
	e^{-s\tau}\K_{41} & 0 & e^{-s\tau}\K_{43} & \K_{44} & 0 \\
	e^{-s\tau}\K_{51} & e^{-s\tau}\K_{52} & e^{-s\tau}\K_{53} & e^{-s\tau}\K_{54} & \K_{55}}
\]}%
\blue{where $\K_{ij} \in \proper$.}
So each agent may use its local measurements with no delay, and measurements from its ancestors with a delay of $\tau$. An output-feedback policy $u = \K y$ (internally) \emph{stabilizes} $\Pp$ if
\[
\bmat{I & -\Pp_{22} \\ -\K & I}^{-1} \in \Hinf.
\]
For further background on stabilization, we refer the reader to \cite{ZDG,DP}. We consider the problem of finding a structured controller that is stabilizing and minimizes the $\Htwo$ norm of the closed-loop map. Specifically, we seek to
\begin{equation}\label{opt}
\begin{aligned}
\underset{\K}{\minimize}
\qquad & \normm{\F_l(\Pp,\K)}_2^2 \\
\subject \qquad & \K \in \Stau
\text{ and $\K$ stabilizes $\Pp$.}
\end{aligned}
\end{equation}

\blue{In the remainder of this section, we list our technical assumptions and define control and estimation gains that will appear in our solution. The assumptions we make ensure that relevant estimation and control subproblems are non-degenerate. We make no assumptions regarding the open-loop stability of $\Pp$. }

\blue{\begin{assumption}[\textbf{System assumptions}]
	\label{Ass:System}
	For the $N$ interacting agents, the Riccati assumptions defined in \cref{defn:Riccati} hold for $(A,B_2,C_1,D_{12})$ and for $(A_{ii}^\tp ,C_{2_{ii}}^\tp ,B_{1_{ii}}^\tp ,D_{{21}_{ii}}^\tp )$ for all $i\in[N]$.
\end{assumption}}

\blue{\begin{defn}[\textbf{Riccati assumptions}]\label{defn:Riccati}
Matrices $(A,B,C,D)$ satisfy the \emph{Riccati assumptions}~\cite{kim2015explicit,mirkin2012h2} if:
\begin{enumerate}[R1.] 
	\item $D^\tp D\succ 0$. \label{defn:Riccati_1}
	\item $(A,B)$ is stabilizable.\label{defn:Riccati_2}
	\item $\begin{bmatrix}A-j\omega I&B\\C&D\end{bmatrix}$ has full column rank for all $\omega \in \R$.\label{defn:Riccati_3}
\end{enumerate}
If the Riccati assumptions hold, there is a unique stabilizing solution for the corresponding algebraic Riccati equation. We write this as $(X,F)=\ric(A,B,C,D)$. Thus, $X \succ 0$ satisfies
\if\MODE1
\begin{multline*}
	A^\tp X+XA+C^\tp C \\-(XB+C^\tp D)(D^\tp D)^{-1}(B^\tp X+D^\tp C) = 0,
\end{multline*}
\else
\begin{equation*}
	A^\tp X+XA+C^\tp C-(XB+C^\tp D)(D^\tp D)^{-1}(B^\tp X+D^\tp C) = 0,
\end{equation*}
\fi
with $A+BF$ Hurwitz and $F\defeq -(D^\tp D)^{-1}(B^\tp X + D^\tp C)$.
\end{defn}}
\subsubsection{Riccati equations}\label{App:ARE}
The algebraic Riccati equations (AREs) corresponding to the centralized linear quadratic regulator (LQR) and Kalman filtering are
\begin{subequations}\label{eq:ARE_cen}
	\begin{align}
		(X_{\textup{cen}},F_{\textup{cen}}) &\defeq
		\ric(A,B_{2},C_{1},D_{12}),\label{eq:ARE_cen:1} \\
		(Y_{\textup{cen}},L_{\textup{cen}}^\tp) &\defeq
		\ric(A^\tp,C_{2}^\tp,B_{1}^\tp,D_{21}^\tp).\label{eq:ARE_cen:2}
	\end{align}
\end{subequations}
Consider controlling the descendants of Agent $i$ using only measurements $y_i$. The associated four-block plant is
\begin{align}
	{\Pp}_i \defeq
	\left[\begin{array}{c c}
		{\Pp}_{11_{:i}} & {\Pp}_{12_{:\underline{i}}}\\
		{\Pp}_{21_{ii}} & {\Pp}_{22_{i\underline{i}}}\end{array}\right] \defeq
	\left[\begin{array}{c|cc}
		A_{\underline{ii}} & B_{1_{\underline{i}i}} & B_{2_{\underline{ii}}} \\[2pt] \hlinet
		{C}_{1_{:\underline{i}}} & 0 & D_{12_{:\underline{i}}} \\
		C_{2_{i\underline{i}}} & D_{21_{ii}} & 0\end{array}\right], \label{Pi}
\end{align} and we define the corresponding ARE solutions as
\begin{subequations}\label{eq:ARE_decen_nodel}
	\begin{align}
		(X^i,F^i) &\defeq
		\ric(A_{\underline{ii}},B_{2_{\underline{ii}}},C_{1_{:\underline{i}}},D_{12_{:\underline{i}}}), \label{eq:ARE_decen_nodel:1}\\
		(Y^i,{L^i}^\tp) &\defeq
		\ric(A_{ii}^\tp,C_{2_{ii}}^\tp,B_{1_{ii}}^\tp,D_{{21}_{ii}}^\tp).\label{eq:ARE_decen_nodel:2}
	\end{align}
\end{subequations}
Note that the block-diagonal structure of the estimation subproblems implies $Y_\textup{cen} = \blkdiag(\{Y^i\})$ and $L_{\textup{cen}} = \blkdiag(\{L^i\})$.
\blue{Existence of the matrices defined in \eqref{eq:ARE_cen} and \eqref{eq:ARE_decen_nodel} follows from \cref{Ass:System} and the fact that $A$, $B_1$, $B_2$, $C_2$, and $D_{21}$ are block-diagonal.}
If we apply the loop-shifting transformation $\Gamma$ described in \cref{sec:delay} and \cref{fig:Mirkin_Gamma}, we obtain the modified plant
\begin{align*}
	\tilde{\Pp}_i \defeq
	\left[\begin{array}{c c}
		\tilde{\Pp}_{11_{:i}} & \tilde{\Pp}_{12_{:\underline{i}}}\\
		{\Pp}_{21_{ii}} & \tilde{\Pp}_{22_{i\underline{i}}}\end{array}\right] \defeq
	\left[\begin{array}{c|cc}
		A_{\underline{ii}} & B_{1_{\underline{i}i}} & \tilde B_{2_{\underline{ii}}} \\[2pt] \hlinet
		\tilde {C}_{1_{:\underline{i}}} & 0 & D_{12_{:\underline{i}}} \\
		C_{2_{i\underline{i}}} & D_{21_{ii}} & 0\end{array}\right].
\end{align*} This modified plant has the same estimation ARE as in \eqref{eq:ARE_decen_nodel:2}, but a new control ARE, which we denote
	\begin{align}
		(\tilde X^i,\tilde F^i) &\defeq
		\ric(A_{\underline{ii}},\tilde B_{2_{\underline{ii}}},\tilde C_{1_{:\underline{i}}},D_{12_{:\underline{i}}}),\label{eq:ARE_decen_del:1}
	\end{align}
\blue{Existence of the matrices defined in \eqref{eq:ARE_decen_del:1} also follows from \cref{Ass:System} \cite[Lem.~4 and Rem.~1]{mirkin2012h2}.}

\section{Optimal Controller}\label{sec:result_consolidated}

We now present our solution to the structured optimal control problem described in \cref{sec:problem_statement}. We begin with a convex parameterization of all structured stabilizing controllers.

\subsection{Parameterization of stabilizing controllers}\label{Sec:stabilization}

This parameterization is similar to the familiar state-space parameterization of all stabilizing controllers \cite{ZDG,DP}, but with an additional constraint on the parameter $\Q$ to enforce the required controller structure.

\begin{lem}\label{lem:stabilizing_controllers}
\blue{Consider the structured optimal control problem described in \cref{sec:problem_statement}}with $\Pp$ given by \eqref{eq:fourblock} and suppose \cref{Ass:System} holds.
Pick $F_d$ and $L_d$ block-diagonal such that $A+B_2 F_d$ and $A+L_d C_2$ are Hurwitz.
The following are equivalent:

\begin{enumerate}[(i)]
\item $\K \in \Stau$ and $\K$ stabilizes $\Pp$.
\item $\K = \F_l(\J,\Q)$ for some $\Q \in \Hinf \cap \Stau$, where
	\begin{align}\label{nominal_Jd}
	 \J \defeq \left[\begin{array}{c|cc}
		A+B_2F_d+L_dC_2 & -L_d & B_{2} \\ \hlinet
		F_d&0&I \\
		-C_{2} & I & 0\end{array}\right].
    \end{align}
\end{enumerate}
\end{lem}
\medskip
\begin{proof}
A similar approach was used in \cite[Thm.~11]{lessard2013structured} to parameterize the set of stabilizing controllers when $\K\in\Ss_0$ (no delays). In the absence of the constraint $\K\in\Stau$, the set of stabilizing controllers is given by $\set{\F_l(\J,\Q)}{\Q\in\Hinf}$ \cite[Thm.~12.8]{ZDG}. It remains to show that $\K\in\Stau$ if and only if $\Q\in\Stau$.
\blue{Expanding the definition of the lower LFT, we have
\begin{equation}\label{KQ1}
\K = \J_{11} + \J_{12}\Q \left( I - \J_{22}\Q \right)^{-1} \J_{21}.
\end{equation}
The matrices $A$, $B_2$, $C_2$, $F_d$, $L_d$ are block-diagonal, so $\J_{ij}$ is block-diagonal and therefore $\J_{ij} \in \Stau$.
The delays in our graph satisfy the triangle inequality, so $\Stau$ is closed under multiplication (whenever the matrix partitions are compatible). Moreover, $\Stau$ is quadratically invariant with respect to $\J_{22}$ \cite{rotkowitz2010convexity}. Therefore, if $\Q\in\Stau$, then $\Q \left( I - \J_{22}\Q \right)^{-1} \in \Stau$ \cite{rotkowitz2010convexity,lessard2014algebraic}, and we conclude from \eqref{KQ1} that $\K\in\Stau$.}
Applying the inversion property of LFTs, we have $\Q = \F_u(\J^{-1},\K)$. Now
\[
\J^{-1} =\left[\begin{array}{c|cc}
	   A & B_2 & -L_d \\[2pt] \hlinet
	   C_2&0&I \\
	   -F_d & I & 0\end{array}\right],
\]
so we can apply a similar argument to the above to conclude that $(\J^{-1})_{ij} \in \Stau$ and $\K\in\Stau \implies \Q\in\Stau$.
\end{proof}

We refer to $\Q$ in \cref{lem:stabilizing_controllers} as the \emph{Youla parameter}, due to its similar role as in the classical Youla parameterization~\cite{youla1976modern}.

\begin{rem}
Although the problem we consider is quadratically invariant (QI), the existing approaches for convexifying a general QI problem \cite{rotkowitz2005characterization}
or even a QI problem involving sparsity and delays \cite{rotkowitz2010convexity} require strong assumptions, such as $\Pp_{22}$ being stable or strongly stabilizable. Due to the particular delay structure of our problem, the parameterization presented in \cref{lem:stabilizing_controllers} does not require any special assumptions and holds for arbitrary (possibly unstable) $\Pp$.
\end{rem}

\begin{rem}\label{nominalcontrol_desc}
In the special case where $A$ is Hurwitz (so $\Pp$ is stable), we can substitute $F_d=0$ and $L_d=0$ in~\eqref{nominal_Jd} to obtain a simpler parameterization of stabilizing controllers.
\end{rem}

Using the parameterization of \cref{lem:stabilizing_controllers}, we can rewrite the synthesis problem \eqref{opt} in terms of the Youla parameter $\Q$. After simplification, we obtain the convex optimization problem
\begin{equation}\label{opt2}
\begin{aligned}
	\underset{\Q}{\minimize}
	\qquad & \normm{\Tx_{11} + \Tx_{12} \Q \Tx_{21}}_2^2 \\
	\subject \qquad & \Q \in \blue{\Hinf} \cap \Stau.
\end{aligned}
\end{equation}
where 
$\Tx =\bmat{\Tx_{11} & \Tx_{12} \\ \Tx_{21} & 0}$ 
\begin{align}\label{T}
	=\squeezemat{1pt}{\left[\begin{array}{cc|cc}
		A+B_2F_d & -B_2F_d & B_1& B_{2} \\
		0&A+L_dC_2&B_1+L_dD_{21}&0\\ \hlinet
		C_1+D_{12}F_d&-D_{12}F_d&0&D_{12} \\
		0&C_2 & D_{21} & 0\end{array}\right]}.
\end{align}

\begin{rem}
	The convex problem \eqref{opt2}--\eqref{T} is similar to its unstructured counterpart \cite[Thm.~12.16]{ZDG}, except we have the additional constraint $\Q\in\Stau$ on the Youla parameter.
\end{rem} 

\begin{rem}
	\blue{We use $L\defeq L_{\textup{cen}}=L_d=\blkdiag(\{L^i\})$ throughout the rest of the article. This choice of $L$ yields a $\Q_{\textup{opt}}$ with reduced state dimension and simplifies our exposition.}
\end{rem}


\subsection{Optimal controller without delays}\label{sec:result_1}

When there are no processing delays ($\tau=0$), the optimal structured controller is rational. We now provide an explicit state-space formula for this optimal $\K$.

\begin{thm} \label{thm:3}
\blue{Consider the structured optimal control problem described in \cref{sec:problem_statement}} and suppose \cref{Ass:System} holds. Choose a block-diagonal $F_d$ such that $A+B_2 F_d$ is Hurwitz. 
A realization of the $\Qopt$ that solves \eqref{opt2} in the case $\tau=0$ is
\begin{align}\label{Qopt1}
	\Qopt={\left[\begin{array}{c|c}%
			\bar{A}+\bar{B}\bar{F}&-\bar{L}\bar{\1}_p\\\hlinet \bar{\1}_m^\tp(\bar{F}-\bar{F}_d)&0\end{array}\right]}
\end{align}
and a corresponding $\K_\textup{opt}$ that solves \eqref{opt} is
\begin{align}\label{Kopt1}
	\K_{\textup{opt}} =\left[\begin{array}{c|c}%
	\bar{A}+\bar{B}\bar{F}+\bar{L}\bar{C}\bar{\1}_n\bar{\1}_{n}^\tp&-\bar{L}\bar{\1}_p\\\hlinet \bar{\1}_m^\tp \bar{F}&0\end{array}\right].
\end{align}
In \eqref{Qopt1}--\eqref{Kopt1}, we defined the new symbols
     \begin{gather*}
     	\bar{A} \!\defeq\! I_{N} \!\otimes\! A, \hspace{8pt}
     	\bar{B} \!\defeq\! I_N \!\otimes\! B_2, \hspace{8pt}
     	\bar{C} \!\defeq\! I_N \!\otimes\! C_2, \hspace{8pt}
		 \bar{F}_d \!\defeq\! I_{N} \!\otimes\! F_d,\\
		 \bar{\1}_n \defeq 1_N\otimes I_n, \qquad
     	\bar{\1}_m \defeq 1_N\otimes I_m, \qquad
     	\bar{\1}_p \defeq 1_N\otimes I_p.
     \end{gather*}
Matrices $\bar L$ and $\bar F$ are block-diagonal concatenations of zero-padded LQR and Kalman gains for each agent. Specifically,
$\bar{F} \defeq \blkdiag(\{E_{m_{\underline{i}}}F^i E_{n_{\underline{i}}}^\tp\})$ and
$\bar{L} \defeq \blkdiag(\{E_{n_{i}}L^iE_{p_{i}}^\tp\})$
for all $i\in [N]$, where $F^i$ and $L^i$ are defined in \eqref{eq:ARE_decen_nodel}.
\end{thm}
\begin{proof}
	See \cref{App:B}.
\end{proof}

\begin{rem}
The optimal controller \eqref{Kopt1} can also be expressed explicitly in terms of the adjacency matrix; see for example \cite{shah2013cal, kashyap2019explicit}. We opt for the realization  \eqref{Kopt1} as this expression generalizes more readily to the case with delays.
\end{rem}

\begin{rem}
Since agents can act as relays, any cycles in the communication graph can be collapsed and the associated nodes can be aggregated \blue{when there are no delays}. For example, the graph of \cref{fig:cartoon} would become the \blue{four-node diamond graph $\{1\} \to \{3,4\} \to \{5\}$, and $\{1\} \to \{2\} \to \{5\}$.} So in the delay-free setting, there is no loss of generality in assuming the communication graph is acyclic.
\end{rem}

\begin{rem}
	Although the optimal $\Qopt$ \eqref{Qopt1} and associated $\J$ \eqref{nominal_Jd} depend explicitly on $F_d$, the optimal $\K_{\textup{opt}}$ \eqref{Kopt1} does not. 
\end{rem}

\subsection{Optimal controller with delays}\label{sec:result_2}
In this section, we generalize \cref{thm:3} to include \blue{an arbitrary but fixed processing delay $\tau > 0$}.
To this end, we introduce a slight abuse of notation to aid in representing non-rational transfer functions. We generalize the notation of \eqref{ssnotation} to allow for $A,B,C,D$ that depend on $s$. So we write:
\[
\stsp{A(s)}{B(s)}{C(s)}{D(s)} \defeq D(s) + 	C(s) \left( s I - A(s) \right)^{-1} B(s).
\]

\begin{thm}
	\label{thm:5}
	Consider the setting of \cref{thm:3}.
	The transfer function of  $\Qopt\in\blue{\Hinf}\cap\Stau$ that solves \eqref{opt2} for any $\tau\geq 0$ is
	\begin{equation}\label{Qopt2_aliter}
		\Qopt =
\squeezemat{1pt}		\left[\begin{array}{cc|c}
			\bar{A}\!+\!\bar{L}\bar{C}&\tilde{B}\tilde{F}\!-\!\bar{L}\bar{\Pi}_b\tilde{F}\!-\!\bar{B}\bar{\Pi}_u\tilde{F} & 0\\
			\bar{L}\bar{C}&\bar{A}\!+\!\tilde{B}\tilde{F}\!-\!\bar{L}\bar{\Pi}_b\tilde{F} & -\bar{L}\bar{\1}_p \\ \hlinet
			\bar{\1}_m^\tp\bar{\Lambda}_m\bar{F}_d&\bar{\1}_m^\tp\bar{\Lambda}_m(\bar{\Pi}_u\tilde{F}\!-\!\bar{F}_d) & 0
		\end{array}\right]
	\end{equation}
and a corresponding $\K_\textup{opt}$ that solves \eqref{opt} is
\begin{equation}\label{Kopt1b}
	\K_{\textup{opt}} =
	\left[\begin{array}{c|c}
		\bar{A}\!+\!\tilde{B}\tilde{F}\!+\!\bar{L}\bar{C}\bar{\1}_n\bar{\1}_{n}^\tp\bar{\Lambda}_n\!-\!\bar{L}\bar{\Pi}_b\tilde{F} & -\bar{L}\bar{\1}_p\\ \hlinet
		\bar{\1}_m^\tp\bar{\Lambda}_m \bar{\Pi}_u\tilde{F} & 0
	\end{array}\right],
\end{equation}
	where $\bar{A}$, $\bar{L}$, $\bar{F}_d$, $\bar{\1}_n$, $\bar{\1}_m$, $\bar{\1}_p$,  are defined in \cref{thm:3}. The remainder of the symbols are defined as follows. We apply the loop-shifting transformation $(\tilde{\Pp}_i,\Pi_{u_i},\Pi_{b_i})=\Gamma(\Pp_{i},\Lambda_m^i)$, where $\Pp_{i}$, $\tilde{\Pp}_i$, $\tilde{F}^i$ are defined in \cref{App:ARE}, and 
	\begin{gather*}
		\tilde{F}\defeq\blkdiag(\{E_{m_{\underline{i}}}\tilde{F}^iE_{n_{\underline{i}}}^\tp\}),\quad
		\bar{\Pi}_{b}\defeq\blkdiag(\{E_{p_{\underline{i}}}\Pi_{b_i}E_{m_{\underline{i}}}^\tp\}),\\
		\tilde{B}\defeq\blkdiag(\{E_{n_{\underline{i}}}\tilde{B}_{2_{\underline{ii}}}E_{m_{\underline{i}}}^\tp\}),\quad 
		\bar{\Pi}_{u}\defeq\blkdiag(\{E_{m_{\underline{i}}}\Pi_{u_i}E_{m_{\underline{i}}}^\tp\}),\\
		\bar{\Lambda}_{k}\defeq\blkdiag(\{E_{k_{\underline{i}}}\Lambda_{k}^iE_{k_{\underline{i}}}^\tp\}), \quad \textit{for} \; k\in\{m,n\}.
	\end{gather*}
\end{thm}
\medskip
\begin{proof}
	See \cref{App:D}.
\end{proof}

The transfer matrices $\Q_\textup{opt}$ in \eqref{Qopt2_aliter} and $\K_\textup{opt}$ in \eqref{Kopt1b} are not rational, due to the presence of the FIR blocks $\bar{\Pi}_u$, $\bar{\Pi}_b$, and delay blocks \blue{$\bar{\Lambda}_m$ and $\bar{\Lambda}_n$}. Consequently, we cannot write standard state-space realizations as in \cref{thm:3}. When $\tau=0$, we have $\bar{\Pi}_u=I$, $\bar{\Pi}_b=0$, $\bar\Lambda_m = I$, $\tilde F = \bar F$, and $\tilde B = \bar B$, and we recover the results of \cref{thm:3}.


\section{Agent-level controllers}\label{sec:result_3}
\begin{figure*}
	\centering
	\includegraphics{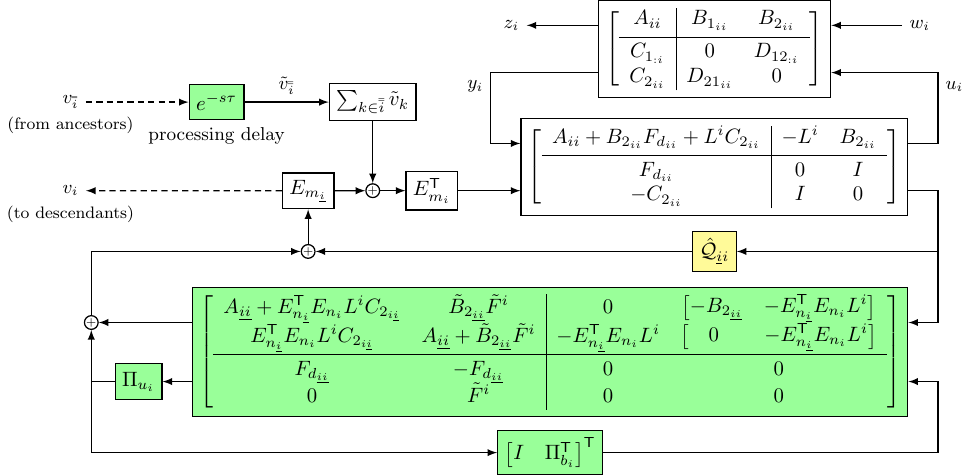}
	\caption{Agent-level implementation of all structured stabilizing controllers, parameterized by $\hat\Q\in\blue{\Hinf}\cap\Ss_0$. Here, $F_d$ is any block-diagonal matrix such that \blue{$A_{ii} + B_{2_{ii}} F_{d_{ii}}$} is Hurwitz.
	The $\Htwo$-optimal controller is achieved when $\hat{\Q}=0$, and results in the simplified diagram of \cref{fig:agent-level-opt}. The blocks that depend on the processing delay $\tau$ are colored in green. All symbols are defined in \cref{thm:6}.}
	\label{fig:agent-level-subopt}
\end{figure*}

\blue{The optimal controller presented in \cref{thm:3}} is generally not minimal. For example, $\K_\textup{opt}$ in \eqref{Kopt1} has a state dimension of $Nn$, which means a copy of the global plant state for each agent. However, if we extract the part of $\K_\textup{opt}$ associated with a particular agent, there is a dramatic reduction in state dimension. So in a distributed implementation of this controller, each agent would only need to store a small subset of the controller's state. \blue{A similar reduction exists for the optimal controller for the delayed problem presented in \cref{thm:5}.}

Our next result presents \blue{reduced implementations} for these \emph{agent-level} controllers and characterizes the information each agent should store and communicate with their neighbors. We find that Agent $i$ simulates its descendants' dynamics, and so has dimension $n_{\underline i}$, which is \blue{at least $N$ times smaller} than the dimension $Nn$ of the aggregate optimal controller from \cref{thm:3}.

\begin{figure}[ht]
	\centering
	\includegraphics{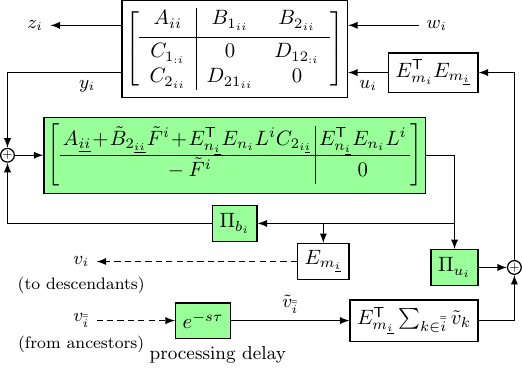}
	\caption{Agent-level implementation of the $\Htwo$-optimal controller with processing delays. This is the result of setting $\hat{\Q}=0$ in \cref{fig:agent-level-subopt}. The blocks that depend on the processing delay $\tau$ are colored in green. All symbols are defined in \cref{thm:6}.}
	\label{fig:agent-level-opt}
\end{figure}

\begin{thm}\label{thm:6}
	Consider the setting of \cref{thm:3} with $\tau\geq0$. The agent-level implementation of all structured stabilizing controllers, parameterized by $\hat \Q\in\blue{\Hinf} \cap \Ss_0$, is shown in \cref{fig:agent-level-subopt}. Here, 
	the optimal controller is achieved when $\hat \Q=0$. In this case, we obtain the simpler structure of \cref{fig:agent-level-opt}. All symbols used  are defined in \cref{thm:3,thm:5}.
\end{thm}

\begin{proof}
	See \cref{App:E}.
\end{proof}

\subsection{Interpretation of optimal controller} 

\cref{fig:agent-level-subopt} shows that Agent $i$ transmits the same signal $v_{i}$ to each of its strict descendants. When an agent receives the signals \blue{$v_{\bar{\bar{i}}}$} from its strict ancestors \blue{$\bar{\bar{i}}$}, it selectively extracts and sums together certain components of the signals. To implement the optimal controller, each agent only needs to know the dynamics and topology of its descendants.

If the network has the additional property that there is at most one directed path connecting any two nodes\footnote{Also known as a \emph{multitree} or a \emph{diamond-free poset}.}, then the communication scheme can be further simplified. Since Agent $i$'s decision $u_i$ is a sum of terms from all ancestors, but each ancestor has exactly one path that leads to $i$, the optimal controller can be implemented by transmitting all information to \emph{immediate descendants only} and performing recursive summations. This scheme is illustrated for a four-node chain graph in \cref{fig:multitree_transmission}.\looseness=-1

\begin{figure}[htb]
	\centering
	\includegraphics{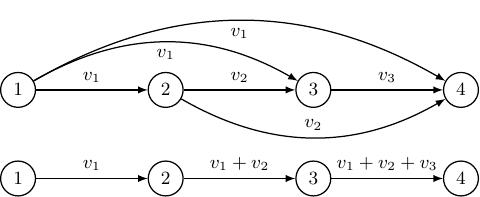}
	\caption{Four-agent chain graph with standard broadcast (top) and efficient immediate-neighbor implementation (bottom), which is possible because this graph is a multitree.}
	\label{fig:multitree_transmission}
\end{figure}

\begin{rem}
	The agent-level controller from \cref{fig:agent-level-opt} can be represented as the combination of an observer with transfer matrix $\mathcal{T}_{\underline{ii}}\defeq (sI-A_{\underline{ii}}-E_{n_{\underline{i}}}^\tp E_{n_{i}}L^i C_{2_{i\underline{i}}})^{-1}$, and a regulator with an LQR gain $\tilde{F}^i$ in \cref{fig:agent-level-opt_v2}. This yields a separation structure reminiscent of standard LQG theory~\cite{ZDG}. 
\end{rem}

\begin{figure}[ht]
	\centering
	\includegraphics{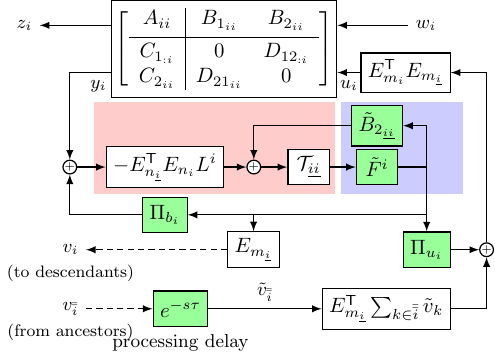}
	\caption{Agent-level implementation of the $\Htwo$-optimal controller with processing delays, featuring an observer (red) and regulator (blue) separation structure. Here $\mathcal{T}_{\underline{ii}}\defeq(sI-A_{\underline{ii}}-E_{n_{\underline{i}}}^\tp E_{n_{i}}L^i C_{2_{i\underline{i}}})^{-1}$ is the transfer matrix of the observer dynamics.}
	\label{fig:agent-level-opt_v2}
\end{figure}

\begin{rem}
	Compared to the architecture proposed in \cite[Fig.~4]{kashyap2020agent}, 
	the agent-level optimal controller in \cref{fig:agent-level-opt} is more efficient because each agent transmits a single vector $v_{i}$ to its descendants, instead of two.  
\end{rem}

\begin{rem}
	The controller in \cref{fig:agent-level-opt} has the form of a feed-forward Smith predictor, similar to \cref{fig:Mirkin_Gamma} (bottom left). The FIR block $\Pi_{u_i}$ compensates for the effect of adobe delay. Similarly, the FIR block $\Pi_{b_i}$ resembles the internal feedback in traditional dead-time controllers.
\end{rem}


\section{Characterizing the cost}\label{sec:result_4}

In this section, we characterize the cost of any structured stabilizing controller. The cost is defined as $J\defeq \normm{\F_l(\Pp,\K)}_2^2=\normm{\Tx_{11} + \Tx_{12} \Q \Tx_{21}}_2^2$, where $\K$ is feasible for~\eqref{opt} or equivalently, $\Q=\F_u(\J^{-1},\K)$ is feasible for~\eqref{opt2} (see \cref{lem:stabilizing_controllers}). We show how to interpret the cost in different ways, and how to compute it efficiently. We illustrate our result using an example with $N=4$ agents.\looseness=-1

\begin{figure}[ht]
	\centering
	\includegraphics{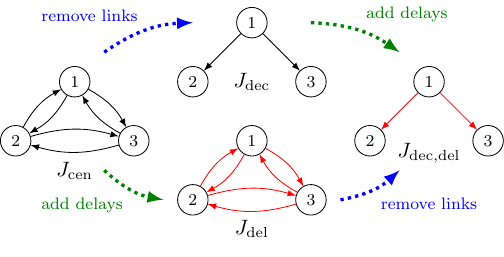}\vspace{-5mm}
	\caption{\blue{Hierarchy of optimal costs for different communication patterns in a three-agent example. Additional cost is incurred if links are removed (blue dotted arrows), or if processing delay is added (green dottted arrows). Delayed edges are red. In this example, $\Dcen \leq \Ddecen \leq \Ddecdel$ and $\Dcen \leq \Dcendel \leq \Ddecdel$ but $\Ddecen$ and $\Dcendel$ are not comparable.}}\vspace{-5mm}
	\label{fig:cost_interp}
\end{figure}
\begin{figure}[ht]
	\centering
	\includegraphics{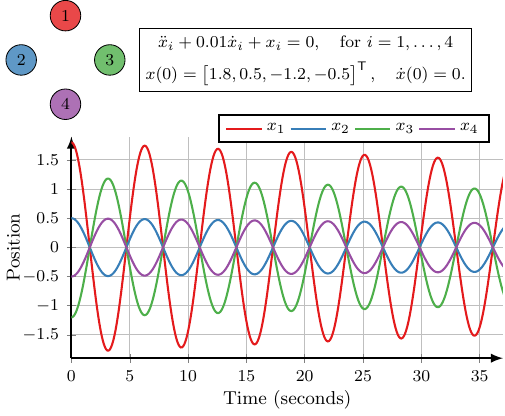}
	\caption{\blue{Open-loop zero-input response of a network of four lightly damped oscillators.}}
	\label{fig:synchronization_openloop}
\end{figure}
\begin{figure}[ht]
	\centering
	\includegraphics{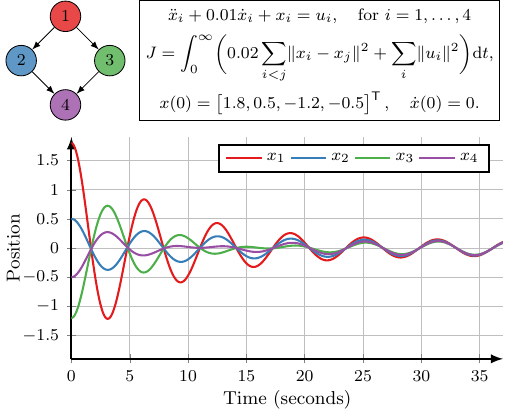}
	\caption{\blue{Closed-loop response of the four-oscillator system from \cref{fig:synchronization_openloop} using the optimal controller from \cref{thm:3} for a diamond-shaped communication graph with no processing delay. The oscillators leverage the communication network to achieve synchronization.}}
	\label{fig:synchronization_example}
\end{figure}
\begin{figure}[ht]
	\centering
	\includegraphics{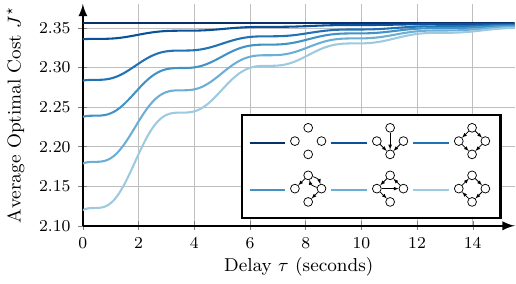}
	\caption{The average optimal cost, as a function of processing delay, for the 4-agent system of \cref{fig:synchronization_example} with different network topologies. For each topology, the cost is an increasing function of the processing delay.}
	\label{fig:cost_vs_delays_sparsity}
\end{figure}

\begin{thm}
	\label{thm:7}
	Consider the setting of \cref{thm:3}.
	The optimal (minimal) costs for the cases: a fully connected graph with no delays, a decentralized graph with no delays, a fully connected graph with delays, and a decentralized graph with delays are:
	\begin{subequations}
	\begin{align}
		\label{Cost:1_2}\Dcen&= \trace(Y_{\textup{cen}}C_1^\tp C_1)+\trace(X_{\textup{cen}}L D_{21} D_{21}^\tp L^\tp),\\
		\label{Cost:2}\Ddecen &=\trace(Y_{\textup{cen}}C_1^\tp C_1)+\trace(X_{\textup{dec}} L D_{21}D_{21}^\tp L^\tp),\\
		\label{Cost:3}\Dcendel &=\trace(Y_{\textup{cen}}C_1^\tp C_1)+\trace(X_{\textup{del}}L D_{21}D_{21}^\tp L^\tp),\\
		\label{Cost:4}\Ddecdel &=\trace(Y_{\textup{cen}}C_1^\tp C_1)+\trace(X_{\textup{dec,del}}L D_{21}D_{21}^\tp L^\tp),
	\end{align}
\end{subequations}
    respectively. If a feasible but sub-optimal $\Q$ is used in any of the above cases, write $\Q_{\Delta}\defeq\Q-\Q_{\textup{opt}}$. The cost of this sub-optimal $\Q$ is found by adding $\DQ\defeq\normm{\Tx_{12}\Q_{\Delta}D_{21}}_2^2$ to \eqref{Cost:1_2}--\eqref{Cost:4}. The various symbols are defined as 
    \begin{gather*}
    X_{\textup{dec}}\defeq \blkdiag(\{X^i(1,1)\}),\quad
    X_{\textup{del}}\defeq\blkdiag(\{\Xi_{c_{\tau}}^i(1,1)\}),\\
	X_{\textup{dec,del}}\defeq\blkdiag(\{\Xi_{\tau}^i(1,1)\}),\qquad \textit{and satisfy}
    \end{gather*}
\begin{subequations}
	\begin{align}
	\label{eq:Ineq:1}	\blkdiag(\{X_{\textup{cen}}(i,i)\})& \preceq X_{\textup{dec}}\preceq X_{\textup{dec,del}}, \\
	\label{eq:Ineq:2}	\blkdiag(\{X_{\textup{cen}}(i,i)\})& \preceq X_{\textup{del}}\preceq X_{\textup{dec,del}}.
	\end{align}
\end{subequations}
 $X_{\textup{cen}}, Y_{\textup{cen}}, F_{\textup{cen}}$, and $L$ are defined in \cref{App:ARE}. $\Xi_{\tau}^i$ and $\Xi_{c_{\tau}}^i$ are defined in \cref{App:F:DRE_1,App:F:DRE_2}, respectively.
\end{thm}
\begin{proof}
	See \cref{App:F}.
\end{proof}
In \eqref{Cost:1_2} we recognize $\Dcen$ as the standard LQG cost (fully connected graph with no delays). Further, there are two intuitive interpretations for \cref{thm:7} that are represented in \cref{fig:cost_interp} 
 for a 3-agents system. The intermediate graph topologies are different, but the starting and ending topologies are equal for both. Along the upper path, $\Ddecen-\Dcen$ is the additional cost incurred due to decentralization alone, and $\Ddecdel-\Ddecen$ is the further additional cost due to delays. Likewise, along the lower path, $\Dcendel-\Dcen$ is the additional cost due to delays alone and $\Ddecdel-\Dcendel$ is the further additional cost due to decentralization.  Finally, $\DQ$ is the additional cost incurred due to suboptimality. \cref{thm:7} unifies existing cost decomposition results for the centralized \cite[\S14.6]{ZDG}, decentralized \cite[Thm.~16]{lessard2015optimal}, and delayed \cite[Prop.~6]{mirkin2012h2} cases.

\begin{rem}
	Delay and decentralization do not contribute independently to the cost. Specifically, the marginal increase in cost due to adding processing delays depends on the graph topology. Likewise, the marginal increase in cost due to removing communication links depends on the processing delay. In other words, $\Dcen+\Ddecdel \neq \Ddecen+\Dcendel$.
\end{rem}
\if\MODE3
\begin{rem}
	There is a dual expression for the cost $\Dcen$ in~\eqref{Cost:1_2}:
	\[
		\Dcen= \trace(X_{\textup{cen}}B_1B_1^\tp)+\trace(Y_{\textup{cen}}F_{\textup{cen}}^\tp D_{12}^\tp D_{12}F_{\textup{cen}}).
	\]
	The corresponding dual expressions for \eqref{Cost:2}--\eqref{Cost:4} are unfortunately more complicated. See \cref{App:Cost_alt} for details.
\end{rem}
\else
\begin{rem}
	There is a dual expression for the cost $\Dcen$ in~\eqref{Cost:1_2}:
	$\Dcen= \trace(X_{\textup{cen}}B_1B_1^\tp)+\trace(Y_{\textup{cen}}F_{\textup{cen}}^\tp D_{12}^\tp D_{12}F_{\textup{cen}})$. The corresponding dual expressions for \eqref{Cost:2}--\eqref{Cost:4} are unfortunately more complicated. See \cref{App:Cost_alt} for details.
\end{rem}
\fi

\subsection{Synchronization example}
\blue{We demonstrate \cref{thm:3} via a simple structured LQG example. We consider $N=4$ identical lightly damped oscillators. The oscillators begin with different initial conditions and the goal is to achieve synchronization. The oscillators} \blue{have identical dynamics defined by the differential equations in \cref{fig:synchronization_openloop,fig:synchronization_example}. \cref{fig:synchronization_openloop} shows the open-loop zero-input response for the four oscillators with given initial conditions. Due to the light damping, the states slowly converge to zero as $t \rightarrow \infty$.}

\blue{\cref{fig:synchronization_example} shows the closed-loop response using the optimal controller from \cref{thm:3} for a diamond-shaped communication network with no processing delay. The controller states are initialized to match the initial state of the plant. Since the observer is an unbiased estimator, the LQG controller replicates the behavior of full-state feedback LQR. \cref{fig:synchronization_example} shows the four oscillators leveraging their shared information to achieve synchronization to a common oscillation pattern.}

In \cref{fig:cost_vs_delays_sparsity}, we use the same system as in \cref{fig:synchronization_example}, but we plot the total average cost as a function of time delay for various network topologies. The highest cost corresponds to a fully disconnected network, while the lowest cost corresponds to a fully connected network. In the limit as $\tau \to\infty$ (infinite processing delay), the cost tends \blue{to that of the fully disconnected case.}

\section{Conclusion}
\label{sec:conclufuture}
We studied a structured optimal control problem where multiple dynamically decoupled agents communicate over a delay network.
Specifically, we characterized the structure and efficient implementation of optimal controllers at the individual agent level.
We now propose some possible future applications for our work.

First, our approach can be readily generalized to treat cases with a combination of processing delays and network latency, where the various delays are heterogeneous but known \blue{\cite{kashyap2023thesis}}.
Next, the observer-regulator architecture elucidated in \cref{fig:agent-level-opt_v2} could also be used to develop heuristics for solving cooperative control problems where the agents' dynamics are nonlinear or the noise distributions are non-Gaussian. Examples could include decentralized versions of the \blue{Extended Kalman Filter or Unscented Kalman Filter.}
Finally, our closed-form expressions for the optimal cost can serve as lower bounds to the cost of practical implementation that have additional memory, power, or bandwidth limitations.

\newpage
\appendix

\if\MODE3
	\section*{Appendices}
	\renewcommand{\thesubsection}{\Alph{subsection}}
\fi

\subsection{Definition of the \texorpdfstring{$\Gamma$}{Gamma} function}\label{sec:appendix_gamma}

The $\Gamma$ function takes in a four-block plant $\Pp$ and adobe delay matrix $\Lambda_m^i$ and returns a transformed plant $\tilde\Pp$ and FIR systems $\Pi_u$, $\Pi_b$. As in \cite{mirkin2012h2}, we first consider the special case where $D_{12}^\tp D_{12} = I$.
The completion operator $\pi_\tau\{\cdot\}$ acts on a rational LTI system delayed by $\tau$ and returns the unique FIR system supported on $[0,\tau]$ that provides a rational completion:
\[
\squeezemat{2pt}
\pi_\tau\Biggl\{\left[\begin{array}{c|c}A&B\\ \hlinet C&0\end{array}\right]\!e^{-s\tau}\Biggr\}
\defeq\left[\begin{array}{c|c}A&e^{-A\tau}\!B\\ \hlinet C&0\end{array}\right]-\left[\begin{array}{c|c}A&B\\ \hlinet C&0\end{array}\right]e^{-s\tau}\!.
\]
The input matrices $B_2$ and $D_{12}$ of $\Pp$ are partitioned according to the blocks of adobe delay matrix $\Lambda_m^i$. So, ${B}_2=\bmat{ B_{2_0} & B_{2_{\tau}} }$, where the two blocks correspond to inputs with delay $0$ and $\tau$, respectively. $D_{12}$ is partitioned in a similar manner.
Define the Hamiltonian matrix
\[
\squeezemat{3pt}
H = \bmat{H_{11} & H_{12} \\ H_{21} & H_{22}} \defeq \bmat{ A \!-\! B_{2_0} D_{12_0}^\tp C_1 & -B_{2_0}B_{2_0}^\tp \\
-C_1^\tp P_\tau C_1 & -A^\tp \!+\! C_1^\tp D_{12_0} B_{2_0}^\tp}\!,
\]
where $P_0 \defeq D_{{12}_0}D_{{12}_0}^\tp$ and $P_{\tau} \defeq I-P_0$, 
and define its matrix exponential as $\Sigma \defeq e^{H\tau}$.
Define the modified matrices
\begin{align*}
	\tilde{B}_2 &\defeq
	\bmat{B_{2_0} & \Sigma_{12}^\tp C_{1}^\tp D_{{12}_{\tau}}+\Sigma_{22}^\tp B_{2_{\tau}}}\\
	\tilde{C}_1 &\defeq \left(P_{\tau} C_{1} + P_0 C_{1} \Sigma_{22}^\tp-D_{{12}_{0}} B_{2_0}^\tp \Sigma_{21}^\tp \right)\Sigma_{22}^{-\tp},	
\end{align*}
where the $\Sigma_{ij}$ are partitioned the same way as the $H_{ij}$. The modified four-block plant output by $\Gamma$ is then
\[
	\tilde{\Pp} \defeq
	\bmat{ \tilde{\Pp}_{11} & \tilde{\Pp}_{12} \\ \Pw & \tilde{\Pp}_{22} } \defeq
	\left[\begin{array}{c|cc}
		A & B_1 & \tilde{B}_2\\ \hlinet
		\tilde{C}_1 & 0 & D_{12} \\
		C_2 & D_{21} & 0 \end{array}\right],
\]
Finally, define the FIR systems
\[
	\bmat{\tilde{\Pi}_u\\ \tilde{\Pi}_b} \defeq
	\pi_{\tau}\!
	\left\{ \left[\begin{array}{cc|c}
		H_{11} & H_{12} & B_{2_{\tau}}\\
		H_{21} & H_{22} & -C_1^\tp D_{12_{\tau}}\\ \hlinet
		D_{12_0}^\tp{C}_1 & B_{2_0} & 0 \\
		C_2 & 0 & 0\end{array}\right]e^{-s\tau}     
	\right\}.
\]FIR outputs of $\Gamma$ are $\Pi_{u} \defeq \bmat{I& \tilde{\Pi}_u\\0&I}$ and $\Pi_b \defeq \bmat{0& \tilde{\Pi}_b}$.

In the general case $D_{12}^\tp D_{12}\neq I$, we can use a standard change of variables to transform back to the case $D_{12}^\tp D_{12} = I$. See \cite[Rem.~2]{mirkin2009loop} for details.

\subsection{Gramian equations}\label{App:Gramian}
	Here we provide the set of Lyapunov equations that are uniquely associated with the multi-agent problem. 
	\begin{lem}\label{lem:X-X_cen}
		Suppose $(X_{\textup{cen}},F_{\textup{cen}})$ and $(X^i,F^i)$ are defined in \eqref{eq:ARE_cen:1} and~\eqref{eq:ARE_decen_nodel:1} respectively. Then $W_X^i\defeq X^i-X_{\textup{cen}_{\underline{ii}}} \succeq 0$ is the unique solution to the Lyapunov equation
		\if\MODE1
		\begin{multline}\label{eq:W_X}
			\hspace{-1mm}(A_{\underline{ii}}+B_{2_{\underline{ii}}}F^i)^\tp W_X^i + W_X^i (A_{\underline{ii}}+B_{2_{\underline{ii}}}F^i)\\ 
			+(E_{m_{\underline{i}}}F^i-F_{\textup{cen}}E_{n_{\underline{i}}})^\tp D_{12}^\tp D_{12} (E_{m_{\underline{i}}}F^i-F_{\textup{cen}}E_{n_{\underline{i}}})=0.
		\end{multline}
		\else
		\begin{equation}\label{eq:W_X}
			(A_{\underline{ii}}+B_{2_{\underline{ii}}}F^i)^\tp W_X^i + W_X^i (A_{\underline{ii}}+B_{2_{\underline{ii}}}F^i)
			+(E_{m_{\underline{i}}}F^i-F_{\textup{cen}}E_{n_{\underline{i}}})^\tp D_{12}^\tp D_{12} (E_{m_{\underline{i}}}F^i-F_{\textup{cen}}E_{n_{\underline{i}}})=0.
		\end{equation}
		\fi 
	\end{lem}
	\begin{proof}
		Left and right multiply the ARE for \eqref{eq:ARE_cen:1} by $E_{n_{\underline{i}}}^\tp$ and $E_{n_{\underline{i}}}$ respectively, and subtract it from~\eqref{eq:ARE_decen_nodel:1}. The result follows from algebraic manipulation and applying the definitions of $F^i$ and $F_{\textup{cen}}$.
		Since the final term in \eqref{eq:W_X} is positive semidefinite and $A_{\underline{ii}}+B_{2_{\underline{ii}}}F^i$ is Hurwitz, it follows that $W_X^i\defeq X^i-X_{\textup{cen}_{\underline{ii}}} \succeq 0$ and is unique. 
	\end{proof}

	We also have a dual analog to \cref{lem:X-X_cen}, provided below.
	\begin{lem}\label{lem:W_Y}
		Consider the setting of \cref{lem:X-X_cen}. There exists a unique $W_Y^i \succeq 0$ that satisfies the Lyapunov equation
		\if\MODE1
		\begin{multline}\label{eq:W_Y}
			(A_{\underline{ii}}+B_{2_{\underline{ii}}}F^i) W_Y^i + W_Y^i (A_{\underline{ii}}+B_{2_{\underline{ii}}}F^i)^\tp\\ 
			+E_{n_{\underline{i}}}^\tp\bar{L}\bar{\1}_p D_{21}D_{21}^\tp\bar{\1}_p^\tp\bar{L}^\tp E_{n_{\underline{i}}}=0.
		\end{multline}
		\else
		\begin{equation}\label{eq:W_Y}
			(A_{\underline{ii}}+B_{2_{\underline{ii}}}F^i) W_Y^i + W_Y^i (A_{\underline{ii}}+B_{2_{\underline{ii}}}F^i)^\tp 
			+E_{n_{\underline{i}}}^\tp\bar{L}\bar{\1}_p D_{21}D_{21}^\tp\bar{\1}_p^\tp\bar{L}^\tp E_{n_{\underline{i}}}=0.
		\end{equation}
		\fi
	\end{lem}
	\begin{proof}
		Since $E_{n_{\underline{i}}}^\tp\bar{L}\bar{\1}_p D_{21}D_{21}^\tp\bar{\1}_p^\tp\bar{L}^\tp E_{n_{\underline{i}}}\succeq 0$ and the matrix $A_{\underline{ii}}+B_{2_{\underline{ii}}}F^i$ is Hurwitz, $W_Y^i \succeq 0$ and is unique. 
\end{proof}

\subsection{Proof of \texorpdfstring{\cref{thm:3}}{Theorem 7}}\label{App:B}

For the case $\tau=0$, 
\blue{we can replace $\Q\in\Hinf\cap \Stau$ by $\Q\in\Htwo\cap\Hinf\cap \Ss_0$ in \eqref{opt2} because the closed-loop map must be strictly proper in order to have a finite $\Htwo$ norm. Since $\Tx_{11}$ is strictly proper, this forces $\Q$ to be strictly proper as well, and hence $\Q\in\Htwo\cap\Hinf$. Further, if $\Q$ is rational, we have $\Q\in\RHtwo$.}
The optimization problem~\eqref{opt2} is a least squares problem with a subspace constraint, so the necessary and sufficient conditions for optimality are given by the normal equations
$\Tx_{12}^\sim ( \Tx_{11} + \Tx_{12} \Q \Tx_{21} ) \Tx_{21}^\sim \in \left( \blue{\RHtwo} \cap \Ss_0 \right)^\perp$ with the constraint that  
$\Q\in\blue{\RHtwo} \cap \Ss_0$.

We can check membership $\F \in (\blue{\RHtwo}\cap\Ss_0)^\perp$ by checking if $\F_{ij} \in \blue{\RHtwo^\perp}$ whenever there is a path $j\to i$. For example, consider the two-node graph $1 \to 2$. Then we have
\if\MODE3
\begin{equation*}
	\blue{\RHtwo} \cap \Ss_0 = \bmat{\blue{\RHtwo} & 0 \\ \blue{\RHtwo} & \blue{\RHtwo}}
	\quad\text{and}\quad
	(\blue{\RHtwo} \cap \Ss_0)^\perp = \bmat{\blue{\RHtwo^\perp} & \Ltwo \\ \blue{\RHtwo^\perp} & \blue{\RHtwo^\perp}}.
	\end{equation*}
\else
\begin{align*}
\blue{\RHtwo} \cap \Ss_0 &= \bmat{\blue{\RHtwo} & 0 \\ \blue{\RHtwo} & \blue{\RHtwo}}
\quad\text{and}\quad\\
(\blue{\RHtwo} \cap \Ss_0)^\perp &= \bmat{\blue{\RHtwo^\perp} & \Ltwo \\ \blue{\RHtwo^\perp} & \blue{\RHtwo^\perp}}.
\end{align*}
\fi
So here, $\F \in (\blue{\RHtwo}\cap\Ss_0)^\perp$ if and only if $\F_{11},\F_{21},\F_{22} \in\blue{\RHtwo^\perp}$.
We will show that the proposed $\Q_\textup{opt}$ in \eqref{Qopt1} is optimal by directly verifying the normal equations.

Substituting $\Q_\textup{opt}$ from \eqref{Qopt1} into $\Tx_{11}+\Tx_{12}\Qopt\Tx_{21}$ with $\Tx_{ij}$ defined in \eqref{T}, we obtain the closed-loop map\looseness=-1
\if\MODE1
\begin{multline}\label{eq:cl_map}
	\Tx_{11}+\Tx_{12}\Qopt\Tx_{21} = \left[\begin{array}{c|c}%
		A_{\textup{cl}}&B_{\textup{cl}}\\\hlinet
		C_{\textup{cl}}&0 \end{array}\right] \\
	\defeq\left[\begin{array}{cc|c}%
			\bar{A}+\bar{B}\bar{F}&-\bar{L}\bar{C}\bar{\1}_n&-\bar{L}\bar{\1}_pD_{21}\\
			0&\blue{A_L}&\blue{B_L}\\\hlinet
			C_1\bar{\1}_n^\tp+D_{12}\bar{\1}_m^\tp\bar{F}&C_{1}&0\end{array}\right],
\end{multline}
\else
\begin{equation}\label{eq:cl_map}
	\Tx_{11}+\Tx_{12}\Qopt\Tx_{21} = \left[\begin{array}{c|c}%
		A_{\textup{cl}}&B_{\textup{cl}}\\\hlinet
		C_{\textup{cl}}&0 \end{array}\right]
	\defeq\left[\begin{array}{cc|c}%
			\bar{A}+\bar{B}\bar{F}&-\bar{L}\bar{C}\bar{\1}_n&-\bar{L}\bar{\1}_pD_{21}\\
			0&\blue{A_L}&\blue{B_L}\\\hlinet
			C_1\bar{\1}_n^\tp+D_{12}\bar{\1}_m^\tp\bar{F}&C_{1}&0\end{array}\right],
\end{equation}
\fi
\blue{where $A_{L}\defeq A+LC_2$ and $B_{L}\defeq B_1+LD_{21}$.}
Next, we show that the controllability Gramian for the closed loop map is block-diagonal.
\begin{lem}\label{lem:clmap+gram}
The controllability Gramian for the closed-loop map \eqref{eq:cl_map} is given by
\if\MODE3
\[\Theta\defeq \blkdiag(\{E_{n_{\underline{i}}}W_Y^i E_{n_{\underline{i}}}^\tp\}_{i\in[N]},Y_{\textup{cen}}),\] 
\else
$\Theta\defeq \blkdiag(\{E_{n_{\underline{i}}}W_Y^i E_{n_{\underline{i}}}^\tp\}_{i\in[N]},Y_{\textup{cen}})$,
\fi
where $Y_{\textup{cen}}$ and $W_Y^i$ are defined in \cref{eq:ARE_cen:2,lem:W_Y}, respectively. In other words, $\Theta \succeq 0$ is the unique solution to 
$A_\textup{cl}\Theta+\Theta A_\textup{cl}^\tp+B_\textup{cl} B_\textup{cl}^\tp=0$.
\end{lem}
\begin{proof}
$A_\textup{cl}$ is Hurwitz and $B_\textup{cl} B_\textup{cl}^\tp \succeq 0$ so the Lyapunov equation has a unique solution and $\Theta \succeq 0$. We can verify the solution by direct substitution using \cref{lem:W_Y} and the ARE associated with \eqref{eq:ARE_cen:2}.
\end{proof}

\cref{lem:clmap+gram} has the following statistical interpretation. If the controlled system \eqref{eq:cl_map} is driven by standard Gaussian noise, its state in these coordinates will have a steady-state covariance $\Theta$, so each block component will be mutually independent.

\subsubsection{Proof of optimality}
Let $\Omega\defeq\Tx_{12}^\sim (\Tx_{11}+\Tx_{12}\Qopt\Tx_{21})\Tx_{21}^\sim$. Substituting $\Q_\textup{opt}$ from \eqref{Qopt1} and using \eqref{eq:cl_map}, we obtain
\begin{align}\label{eq:Omega}
	\Omega &={\left[\begin{array}{ccc|c}%
			-A_{K}^\tp&-C_{K}^\tp C_{\textup{cl}}&0&0\\
			0&A_{\textup{cl}}&B_{\textup{cl}}B_{L}^\tp&B_{\textup{cl}}D_{21}^\tp\\
			0&0&-A_{L}^\tp&-C_2^\tp\\\hlinet
			B_2^\tp&D_{12}^\tp C_{\textup{cl}}&0&0\end{array}\right]}, 
\end{align} where $A_{K}\defeq A+B_2F_d$, $C_{K}\defeq C_1+D_{12}F_d$, and $A_{\textup{cl}}$, $B_{\textup{cl}}$, $C_{\textup{cl}}$, are defined in~\eqref{eq:cl_map}. Apply the state transformation
\begin{align*}
	T = \bmat{
			I& \bmat{\bar{\1}_n^\tp\bar{X} & 0}&\bar{\1}_n^\tp\bar{X}\bar{W}\bar{\1}_p\\
			0&I&\Theta \bar{\1}_p\\
			0&0&I}
\end{align*}  
to \eqref{eq:Omega}, where we defined $\bar W\defeq\blkdiag(\{E_{n_{\underline{i}}}W_Y^iE_{n_{\underline{i}}}^\tp\}_{i \in [N]})$ and $\bar X \defeq \blkdiag(\{E_{n_{\underline{i}}}W_X^iE_{n_{\underline{i}}}^\tp+X_{\textup{cen}}\}_{i \in [N]})$,  and $\Theta$ is defined in \cref{lem:clmap+gram}. \blue{The transformed $\Omega$ is
\begin{align*}
	\Omega={\left[\begin{array}{cccc|c}%
			-A_{K}^\tp&{\star}_1&\star&\star&\star\\
			0&\bar{A}+\bar{B}\bar{F}&-\bar{L}\bar{C}\bar{\1}_n&{\star}_2&{\star}_3\\
			0&0&A_L&{\star}_5&{\star}_6\\
			0&0&0&-A_{L}^\tp&-C_2^\tp\\\hlinet
			B_2^\tp&{\star}_4&D_{12}^\tp C_1&\star&0\end{array}\right]},
\end{align*}
where we have defined the symbols
\begin{align*}
	{\star}_1 &\defeq-A_{K}^\tp\bar{\1}_n^\tp\bar{X}\!-\!C_{K}^\tp(C_1\bar{\1}_n^\tp\!+\!D_{12}\bar{\1}_m^\tp\bar{F})\!-\!\bar{\1}_n^\tp\bar{X}(\bar{A}\!+\!\bar{B}\bar{F}) \\
	{\star}_2 &\defeq-\bar{L}\bar{\1}_pD_{21}B_{L}^\tp\!-\!\bar{L}\bar{C}\bar{\1}_nY_{\textup{cen}}\!+\!(\bar{A}\!+\!\bar{B}\bar{F})\bar{W}\bar{\1}_p\!+\!\bar{W}\bar{\1}_pA_{L}^\tp \\
	{\star}_3&\defeq-\bar{L}\bar{\1}_pD_{21}^\tp+\bar{W}\bar{\1}_pC_2^\tp \\
	{\star}_4&\defeq D_{12}^\tp (C_1\bar{\1}_n^\tp+D_{12}\bar{\1}_m^\tp\bar{F})+B_2^\tp\bar{\1}_n^\tp\bar{X}\\
	{\star}_5&\defeq A_{L}Y_{\textup{cen}}+B_{L}B_{L}^\tp+Y_{\textup{cen}}A_{L}^\tp\\
	{\star}_6&\defeq B_{L}D_{21}^\tp+Y_{\textup{cen}}C_2^\tp.
\end{align*}
\blue{A $\star$ without subscript denotes an unimportant block.}  
Simplifying using Riccati and Lyapunov equations from \cref{App:ARE} and \cref{App:Gramian} respectively, we get ${\star}_5={\star}_6=0$; the mode $A_L$ is uncontrollable.} Removing it, we obtain 
\begin{align}
	\Omega={\left[\begin{array}{ccc|c}%
			-A_{K}^\tp&\star_1&\star&\star\\
			0&\bar{A}+\bar{B}\bar{F}&{\star}_2&{\star}_3\\
			0&0&-A_{L}^\tp&-C_2^\tp\\\hlinet
			B_2^\tp&{\star}_4&\star&0\end{array}\right]}.
\end{align}
Now consider a block $\Omega_{ij}$ for which there is a path $j\to i$. 
\begin{align}\label{omegij}
	\Omega_{ij} &={\left[\begin{array}{ccc|c}%
			-A_{K_{ii}}^\tp&{\star}_{1_{i:}}&\star&\star\\[1mm]
			0&\bar{A}+\bar{B}\bar{F}&{\star}_{2_{:j}}&{\star}_{3_{:j}}\\[1mm]
			0&0&-A_{L_{jj}}^\tp&-C_{2_{jj}}^\tp\\[1mm]\hlinet
			B_{2_{ii}}^\tp&{\star}_{4_{i:}}&\star&0\end{array}\right]}.
\end{align}
Let $\star_1^k$ and $\star_4^k$ denote the $k^\text{th}$ block column and let $\star_2^k$ and $\star_3^k$ denote the $k^\text{th}$ block row. Algebraic manipulation reveals that
\begin{enumerate}[(i)]
\item If $i\in \underline k$ and $\ell\in \underline k$, then $\star^k_{1_{i\ell}} =  \star^k_{4_{i\ell}} = 0$. \label{xi}
\item If $\ell\notin \underline k$ or $j\notin \underline k$, then $\star^k_{2_{\ell j}}= \star^k_{3_{\ell j}} = 0$. \label{xii}
\end{enumerate}
Consider the $k^\text{th}$ diagonal block of $\bar A + \bar B \bar F$ in \eqref{omegij}, which is $A + E_{n_{\underline k}}B_{2_{\underline k \underline k}} F^k E_{n_{\underline k}}^\tp$. This block is itself block-diagonal; it contains the block $A_{\underline k\underline k}+B_{2_{\underline k \underline k}} F^k$ and smaller blocks $A_{\ell \ell}$ for all $\ell \notin \underline k$. We have three cases.
\begin{enumerate}[1.]
\item If $k \in \bar i$, then for all $\ell \in \underline k$, we have $\star^k_{1_{i\ell}} =  \star^k_{4_{i\ell}} = 0$ from Item \eqref{xi} above, so the mode $A_{\underline k\underline k}+B_{2_{\underline k \underline k}} F^k$ is unobservable.

\item If $k\in\bar i$, but instead $\ell \notin \underline k$, we have $\star^k_{2_{\ell j}}= \star^k_{3_{\ell j}} = 0$ from Item \eqref{xii} above, so the modes $A_{\ell\ell}$ are uncontrollable.

\item If $k \notin \bar i$, then $k \notin \bar j$ because $j\to i$ by assumption. Then from Item \eqref{xii} above, all such modes are uncontrollable.
\end{enumerate}
Consequently every block of $\bar A + \bar B \bar F$ is either uncontrollable or unobservable, leading us to the reduced realization
\begin{align}
	\Omega_{ij} &={\left[\begin{array}{cc|c}%
			-A_{K_{ii}}^\tp&\star&\star\\
			0&-A_{L_{jj}}^\tp&-C_{2_{jj}}^\tp\\\hlinet
			B_{2_{ii}}^\tp&\star&0\end{array}\right]}.
\end{align}
Therefore, $\Omega_{ij} \in \blue{\RHtwo^\perp}$ whenever $j\to i$, as required.

\subsection{Proof of \texorpdfstring{\cref{thm:5}}{Theorem 11}}\label{App:D}

Start with the convexified optimization problem \eqref{opt2}. Based on the structured realization \eqref{T}, we see that $\Tx_{21}$ is block-diagonal. Therefore, the optimal cost can be split by columns:
\[
\normm{\Tx_{11} + \Tx_{12}\Q\Tx_{21}}_2^2
= \sum_{i=1}^N \normm{\Tx_{{11}_{:i}} + \Tx_{{12}_{:i}}\Q_{\underline ii} \Tx_{{21}_{ii}}}_2^2.
\]
Since $\Q\in\blue{\Hinf}\cap\Stau$, we can factor each block column of $\Q$ as $\Q_{\underline ii} = \Lambda_m^i \tilde\Q_{\underline ii}$, where $\tilde\Q_{\underline ii}\in\blue{\Hinf}$ has no structure or delay, and $\Lambda_m^i$ is the adobe delay matrix (defined in \cref{sec:delay}). We can therefore optimize for each block column $\tilde\Q_{\underline ii}$ separately. Thus, each subproblem is to
\begin{equation}\label{opt2x}
		\underset{\tilde\Q_{\underline ii} \in \blue{\Hinf}}{\minimize}
		\qquad \normm{\Tx_{{11}_{:i}} + \Tx_{{12}_{:i}}\Lambda_m^i \tilde\Q_{\underline ii} \Tx_{{21}_{ii}}}_2^2,
\end{equation}
Define $\Tx_i \defeq \bmat{ \Tx_{{11}_{:i}} & \Tx_{{12}_{:\underline i}} \\
\Tx_{{21}_{ii}} & 0 }$. Comparing to \eqref{opt2}--\eqref{T}, we observe that \eqref{opt2x} is a special case of the problem \eqref{opt2}, subject to the transformations $\Pp \mapsto \Pp_i$ (defined in \eqref{Pi}) and $F_d \mapsto F_{d_{\underline i\underline i}}$, $L_d \mapsto E^\tp_{n_{\underline i}} E_{n_i}L^i$, and $\Q \mapsto \Lambda_m^i \tilde \Q_{\underline i i}$. If we define the associated $\J_i$ for this subproblem (according to \eqref{nominal_Jd}), we view the subproblem as that of finding the $\Htwo$-optimal controller for the  plant $\Pp_i$ subject to an adobe input delay, as illustrated in the left panel of \cref{fig:cloop_shifting_2}. The key difference between this problem and \eqref{opt} is that we no longer have a sparsity constraint.

\begin{figure}[ht]
	\centering
	\includegraphics{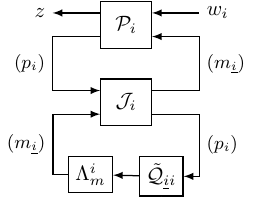}\hspace{5mm}\includegraphics{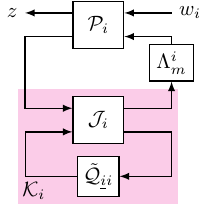}
	\caption{Equivalent subproblems via commuting $\Lambda_m^i$ and $\J_i$. Dimensions of signals are indicated along the arrows.}
	\label{fig:cloop_shifting_2}
\end{figure}

The adobe delay $\Lambda_m^i$ can be shifted to the input channel, shown in the right panel of \cref{fig:cloop_shifting_2}.
This follows from leveraging state-space properties and the block structure of certain blocks of $\J_i$. Examples include $B_{2_{\underline i\underline i}}\Lambda_m^i = \Lambda_n^i B_{2_{\underline i\underline i}}$ and $\Lambda_{n}^i E_{n_{\underline{i}}}^\tp E_{n_i} L^i = E_{n_{\underline{i}}}^\tp E_{n_i} L^i$.

The remainder of the proof proceeds as follows: we define $\K_i$ to be the shaded system in \cref{fig:cloop_shifting_2} (right panel). This is a standard adobe delayed problem, so we can apply the $\Gamma$ transformation illustrated in \cref{fig:Mirkin_Gamma}. Specifically, we define $(\tilde\Pp_i,\Pi_{u_i},\Pi_{b_i}) = \Gamma(\Pp_i,\Lambda_m^i)$, and obtain \cref{fig:cloop_shifting_3}.

\begin{figure}[ht]
	\centering
	\includegraphics{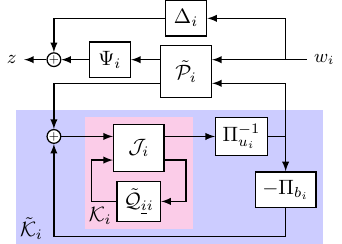}
	\caption{Transformation of the right panel of \cref{fig:cloop_shifting_2} using the loop-shifting transformation illustrated in \cref{fig:Mirkin_Gamma}.}
	\label{fig:cloop_shifting_3}
\end{figure}

By the properties of the loop-shifting transformation discussed in \cref{sec:delay}, the optimal $\tilde K_i$ is found by solving a standard non-delayed LQG problem in the (rational) plant $\tilde \Pp_i$, whose solution is
\begin{align*}
	\tilde{\K}_i &= \left[\begin{array}{c|c}%
		A_{\underline{ii}}+\tilde{B}_{2_{\underline{ii}}}\tilde{F}^i+E_{n_{\underline{i}}}^\tp E_{n_i} L^i C_{2_{i\underline{i}}}& -E_{n_{\underline{i}}}^\tp E_{n_i} L^i \\\hlinet 
		\tilde{F}^i & 0\end{array}\right]\!.
\end{align*}
Inverting each transformation, $\K_i = \Pi_{u_i}\tilde{\K}_i(I-\Pi_{b_i}\tilde{\K}_i)^{-1}$, and we can recover the Youla parameter via $\tilde \Q_{\underline i i} = \F_u(\J_{i}^{-1},\K_i)$, which leads to \eqref{Qt_longeqn}.
\begin{figure*}[ht]
	\begin{equation}\label{Qt_longeqn}
	\tilde{\Q}_{\underline{i}i}=
		\left[\begin{array}{cc|c}%
			A_{\underline{ii}} & {B}_{2_{\underline{ii}}}\Pi_{u_i}\tilde{F}^i& -E_{n_{\underline{i}}}^\tp E_{n_i} L^i\\
			-E_{n_{\underline{i}}}^\tp E_{n_i} L^i C_{2_{i\underline{i}}}& A_{\underline{ii}}+\tilde{B}_{2_{\underline{ii}}}\tilde{F}^i-E_{n_{\underline{i}}}^\tp E_{n_i} L^i\Pi_{b_i}\tilde{F}^i+E_{n_{\underline{i}}}^\tp E_{n_i} L^i C_{2_{i\underline{i}}} & -E_{n_{\underline{i}}}^\tp E_{n_i} L^i\\\hlinet 
			-F_{d_{\underline{ii}}}& \Pi_{u_i}\tilde{F}^i & 0\end{array}\right].
	\end{equation}
\end{figure*}
Now zero-pad, reintroduce delays, and concatenate, to obtain the global Youla parameter \eqref{Qopt2_aliter} via
$\Q_\textup{opt} = \sum_{i=1}^N E_{m_{\underline i}}\Lambda_m^i \tilde\Q_{\underline ii} E_{p_i}^\tp$ and recover the optimal controller \eqref{Kopt1b} via $\K_\textup{opt} = \F_l(\J,\Q_\textup{opt})$.

\subsection{Proof of \texorpdfstring{\cref{thm:6}}{Theorem 12}}\label{App:E}

From \cref{lem:stabilizing_controllers}, the set of sub-optimal controllers is parameterized as $\K = \F_l(\J,\Q)$, where $\Q\in\Stau$. Equivalently, write $\K = \F_l(\J,\Qopt+\Q_\Delta)$, where $\Q_\Delta \in \Stau$ and $\Qopt$ is given in \cref{thm:5}. The controller equation $u = \K y$ can be expanded using the LFT as $\sbmat{u \\ \eta} = \J \sbmat{y \\ v}$ with $v = \Q \eta$. 
If $\J$ has state $\xi$, the state-space equation for $\J$ decouples as
\begin{align*}
\dot\xi_i &= (A_{ii}+B_{2_{ii}} F_{d_{ii}}+L^i C_{2_{ii}})\xi_i - L^i y_i + B_{2_{ii}} v_i, \\
u_i &= F_{d_{ii}} \xi_i + v_i, \\
\eta_i &= -C_{2_{ii}} \xi_i + y_i, \quad\text{for }i=1,\dots,N.
\end{align*}
Note that we replaced $L_{d_{ii}}$ by $L^i$ from \eqref{eq:ARE_decen_nodel:2}. This leads to simpler algebra, but is in principle not required.
Meanwhile, the $\Q$ equation is coupled: $v = (\Qopt + \Q_\Delta) \eta$. Now consider Agent~$i$. Since we are interested in the agent-level implementation, we begin by extracting $u_i$, which requires finding $v_i$. Separate $\Q$ by columns as in \cref{App:D} to obtain
\begin{align}\label{veqn}
v_i &= E_{m_i}^\tp \left(\Qopt+\Q_\Delta\right) \eta \notag\\
&= \sum_{k\in[N]} E_{m_i}^\tp E_{m_{\underline k}}\Lambda_m^k \left(\tilde \Q_{\underline kk} + \hat \Q_{\underline kk}\right)\eta_k \notag \\
&= \left(\tilde \Q_{ii}+\hat \Q_{ii}\right) \eta_i + e^{-s\tau} \sum_{k\in \bar{\bar i}} \left( \tilde \Q_{ik} + \hat \Q_{ik}\right) \eta_k,
\end{align}
where $\tilde\Q_{\underline i i}$ is given in \eqref{Qt_longeqn}, and $\hat \Q\in\Ss_0$ is the delay-free component of $\Q_\Delta$.
A possible distributed implementation is to have Agent~$i$ simulate $\xi_i$ locally. Since $y_i$ is available locally, then so is $\eta_i$. We further suppose Agent~$i$ computes $v_{i,\underline i} \defeq (\tilde \Q_{\underline ii}+\hat\Q_{\underline ii}) \eta_i$ locally. Component $v_{i,i}$ is used locally, while component $v_{i,j}$ for $j\in \underline{\underline i}$ is transmitted to descendant $j$. Each agent then computes $v_i$ by summing its local $v_{i,i}$ with the delayed $e^{-s\tau} v_{i,k}$ received from strict ancestors $k\in\bar{\bar i}$.
The complete agent-level implementation is shown in \cref{fig:agent-level-subopt}.

When $\hat\Q=0$, we recover the optimal controller. In this case, the equations simplify considerably; standard state-space manipulations reduce \cref{fig:agent-level-subopt} to the simpler \cref{fig:agent-level-opt}. It is worth noting that the optimal controller does not depend on the choice of nominal gain $F_d$.

\subsection{Proof of \texorpdfstring{\cref{thm:7}}{Theorem 16}}\label{App:F}

All the estimation, control gains and Riccati solutions used here are defined in \cref{App:ARE}. The additional cost incurred due to suboptimality is $\DQ\defeq \normm{\Tx_{12}\Q_{\Delta}\Tx_{21}}_2^2$ \cite[\S14.6]{ZDG}. Using \cite[Lem.~14.3]{ZDG}, we have $\DQ\defeq \normm{\Tx_{12}\Q_{\Delta}D_{21} }_2^2$. 
\subsubsection{$\Dcen$~\eqref{Cost:1_2}}\label{App:Cost_cen} The optimal cost for a fully connected graph \cite[Thm.~14.7]{ZDG} is
\begin{align*}
	\Dcen&\defeq \normmmmm{\left[\begin{array}{c|c}
			A+B_2F_{\textup{cen}} &B_1\\\hlinet C_1+D_{12}F_{\textup{cen}}&0\end{array}\right]}_2^2+\normmmmm{\left[\begin{array}{c|c}
			A_{\textit{L}} & B_{\textit{L}} \\ \hlinet D_{12}F_{\textup{cen}}&0\end{array}\right]}_2^2,\\
		&= \trace(Y_{\textup{cen}}C_1^\tp C_1)+\trace(X_{\textup{cen}}L D_{21} D_{21}^\tp L^\tp),\\
	&= \trace(X_{\textup{cen}}B_1B_1^\tp)+\trace(Y_{\textup{cen}}F_{\textup{cen}}^\tp D_{12}^\tp D_{12}F_{\textup{cen}}),
\end{align*}
where $A_{\textit{L}}$, $B_{\textit{L}}$ are defined in \cref{App:B} for \eqref{eq:cl_map}.
\subsubsection{$\Ddecen$~\eqref{Cost:2}}\label{App:Cost_decen}
Consider that $\K_{\textup{opt}}$ in \eqref{Kopt1} is a sub-optimal centralized controller for $\normm{\Tx_{11}+\Tx_{12}\Q \Tx_{21}}_2^2$, subject to $\Q\in\blue{\RHtwo}$. Centralized $\Htwo$ theory~\cite{ZDG} implies that $\Ddecen=\Dcen+\Delta$, where $\Delta\defeq \normm{D_{12}\Q_{\textup{you}} D_{21}}_2^2$ and $\Q_{\textup{you}}$ is the centralized Youla parameter. Here, $\Q_{\textup{you}}=\F_u(\J^{-1},\K_{\textup{opt}})$, where
\begin{align*}
	\J^{-1}=\left[\begin{array}{c|cc}
		A & B_2 & -L \\[2pt] \hlinet
		C_2 & 0& I\\
		-F_{\textup{cen}}&I&0
	\end{array}\right].
\end{align*}
After simplifications, we obtain
\begin{align*}
	\Q_{\textup{you}} ={\left[\begin{array}{c|c}%
			\bar{A}+\bar{B}\bar{F}&-\bar{L}\bar{\1}_p\\\hlinet \bar{\1}_m^\tp(\bar{F}-\bar{F}_{\textup{cen}})&0\end{array}\right]}.
\end{align*}
We substitute $\Q_{\textup{you}}$ into the expression for $\Delta$, using $\normm{D_s+C_s(sI-A_s)^{-1}B_s}_2^2=\trace(C_sW_cC_s^\tp)$, where $W_c$ is the controllability Gramian given by Lyapunov equation $A_sW_c+W_cA_s^\tp+B_sB_s^\tp=0$. Based on the \cref{lem:X-X_cen} and using the identity $L_i=E_{n_i}L^iE_{p_i}^\tp$, we evaluate
\begin{align*}
\Delta	&=\sum_{i=1}^{N}\trace(D_{21}^\tp L_i^\tp E_{n_{\underline{i}}} \{X^i-X_{\textup{cen}_{\underline{ii}}}\} E_{n_{\underline{i}}}^\tp L_i D_{21})\\
	&=\trace((\blkdiag(\{X^i(1,1)\})-X_{\textup{cen}})L D_{21}D_{21}^\tp L^\tp).
\end{align*}

We obtain~\eqref{Cost:2} by substituting $\Delta$ into $\Ddecen=\Dcen+\Delta$.
\subsubsection{Alternative formulas for the cost}\label{App:Cost_alt}
We  obtained an alternative formula for $\Dcen$ in \cref{App:Cost_cen}. Similarly, in \cref{App:Cost_decen} for $\Ddecen$, $\normm{D_s+C_s(sI-A_s)^{-1}B_s}_2^2$ is also equal to $\trace(B_sB_s^\tp W_o)$, where $W_o$ is the observability Gramian given by the dual Lyapunov equation $A_s^\tp W_o+W_oA_s+C_s^\tp C_s=0$. Based on \cref{lem:W_Y}, we can evaluate $\Delta=\sum_{i=1}^{N}\trace( D_{12}(E_{m_{\underline{i}}}F_i-F_{\textup{cen}} E_{n_{\underline{i}}}) W_Y^i  (E_{m_{\underline{i}}}F_i-F_{\textup{cen}}E_{n_{\underline{i}}})^\tp D_{12}^\tp)$. Similar alternative formulas exist for~\eqref{Cost:3}, and~\eqref{Cost:4} as well.
\subsubsection{$\Ddecdel$~\eqref{Cost:3}}

We can split the cost in \eqref{opt2} into a sum of $N$ separate terms because $\Tx_{21}$ is block-diagonal. Using \cite[Prop.~6]{mirkin2012h2} on each of these $N$ problems, we write $\Ddecdel$ as a combination of a non-delayed cost $\Ddecen$ and a $\Delta$ incurred by adding delays to that system: 
$\Ddecdel=\Ddecen+\Delta$, where $\Delta \defeq \sum_{i=1}^N \trace(D_{{21}_{ii}}^\tp L_i^\tp E_{n_{\underline{i}}}^\tp(\Xi_{\tau}^i-X^i)E_{n_{\underline{i}}} L_i D_{{21}_{ii}})$. Also, $\Delta = \trace(\blkdiag(\{\Xi_{\tau}^i(1,1)-X^i(1,1)\})L D_{21}D_{21}^\tp L^\tp)$ since $L_i=E_{n_{i}}L^iE_{p_{i}}^\tp$. We obtain \eqref{Cost:3} by substituting $\Delta$ into $\Ddecdel=\Ddecen+\Delta$. See \cref{App:F:DRE_1} below for explanation on $\Xi_{\tau}^i$.

\subsubsection{$\Dcendel$~\eqref{Cost:4}}
Derivation is analogous to that of $\Ddecdel$. See \cref{App:F:DRE_2} below for explanation on $\Xi_{c_{\tau}}^i$. 

\subsubsection{Proofs for~\eqref{eq:Ineq:1}}\label{App:F:DRE_1}
We have $X^i-X_{{\textup{cen}}_{\underline{ii}}}\succeq0$ in \cref{lem:X-X_cen} for all $i\in[N]$. The properties of a positive semi-definite matrix give us $X^i(1,1)-X_{{\textup{cen}}_{\underline{ii}}}(1,1)\succeq0$, and hence $\blkdiag(\{X_{\textup{cen}}(i,i)\}) \preceq X_{\textup{dec}}$. 

Now we define $\Xi_{\tau}^i$ and establish that $\Xi_{\tau}^i-X^i \succeq 0$. The Hamiltonian for the control Riccati equation \eqref{eq:ARE_decen_del:1} is \[
\squeezemat{10pt}
H^i \defeq \bmat{ A_{\underline{ii}} - \tilde{B}_{2_{i\underline{i}}} M^{-1} D_{12_{:i}}^\tp \tilde{C}_{1_{:\underline{i}}} & -\tilde{B}_{2_{i\underline{i}}} M^{-1}\tilde{B}_{2_{i\underline{i}}}^\tp \\
	-\tilde{C}_{1_{:\underline{i}}}^\tp P_\tau \tilde{C}_{1_{:\underline{i}}} & -A_{\underline{ii}}^\tp 
	+ \tilde{C}_{1_{:\underline{i}}}^\tp D_{12_{:i}} M^{{-1}^{\tp}} \tilde{B}_{2_{i\underline{i}}}^\tp}\!,
\]
where $M\defeq D_{12_{:i}}^\tp D_{12_{:i}}$, $P_0 \defeq D_{12_{:i}} M^{-1} D_{12_{:i}}^\tp$ and $P_{\tau} \defeq I-P_0$,
and define the corresponding symplectic matrix exponential as $\Sigma^i \defeq e^{H^i\tau}$. The elements $\Sigma_{22}^i$, $\Sigma_{21}^i$ of this modified $\Sigma^i$ are used to define the $\Xi_{\tau}^i$. For all $i\in[N]$, we define $\Xi_{\tau}^i\defeq\tilde{X}^i-(\Sigma_{22}^{i^{-1}}\Sigma_{21}^i)^\tp$. By solving the associated Differential Riccati Equation (DRE) \cite[Eq.~16]{mirkin2012h2}, we show
$\Xi_{\tau}^i-X^i\succeq0$ \cite[\S4.3]{mirkin2012h2}. This gives us $X_{\textup{dec}}\preceq X_{\textup{dec,del}}$.
\subsubsection{Proofs for~\eqref{eq:Ineq:2}}\label{App:F:DRE_2}
Next we consider the case of a fully connected graph with delays. So Agent $i$'s feedback policy looks like $u_i = \K_{ii}(s)y_i + \sum_{j \in [N]\setminus i} e^{-s\tau}\K_{ij}(s) y_j$. Since we solve for $\Q$ by solving for individual columns $\Q_{\underline{i}i}$, we define the associated state transition matrix for each column as $A_{{\underline{ii}}}^c\defeq \blkdiag(\{A_{ii},A_{\underline{\underline{ii}}}\})$, where $\underline{\underline{i}}=[N]\setminus i$. We define the corresponding  $B_{1_{{\underline{i}i}}}^c$, $ B_{2_{{\underline{ii}}}}^c$ ${C}_{1_{{:\underline{i}}}}^c$, $D_{12_{{:\underline{i}}}}^c$, $C_{2_{{i\underline{i}}}}^c$, and $D_{21_{{ii}}}^c$ in a similar manner.  We also define a centralized $\Xi_{c_{\tau}}^i\defeq\tilde{X}_c^i-(\Sigma_{{22}_c}^{i^{-1}}\Sigma_{{21}_c}^i)^{\tp}$ for each $\Gamma$-modified plant
\begin{align*}
	\tilde{\Pp}_{i}^c \defeq
	\left[\begin{array}{c|cc}
		A_{{\underline{ii}}}^c & B_{1_{{\underline{i}i}}}^c & \tilde B_{2_{{\underline{ii}}}}^c \\[2pt] \hlinet
		\tilde {C}_{1_{{:\underline{i}}}}^c & 0 & D_{12_{{:\underline{i}}}}^c \\
		C_{2_{{i\underline{i}}}}^c & D_{21_{{ii}}}^c & 0\end{array}\right].
\end{align*}
Each individual column $\Q_{\underline{i}i}$ has its own $\tilde\Pp^c_i$ as the associated adobe delay matrix is different. We have a corresponding control ARE $(\tilde X_c^i,\tilde F_c^i) \defeq
	\ric(A_{\underline{ii}}^c,\tilde B_{2_{\underline{ii}}}^c,\tilde C_{1_{:\underline{i}}}^c,D_{12_{:\underline{i}}}^c).$ 
We solve DREs for each $\Xi_{c_{\tau}}^i$ as in \cite[\S V.C]{mirkin2012h2} to obtain $\Xi_{c_{\tau}}^i-X_{{\text{cen}}_{\underline{ii}}}^c\succeq0$ for all $i\in[N]$, where $X_{{\text{cen}}_{\underline{ii}}}^c$ is a reshuffling of $X_{\text{cen}}$ to mirror the ordering of $\underline{i}=\{i,[N]\setminus i\}$. This proves that $\blkdiag(\{X_{\textup{cen}}(i,i)\}) \preceq X_{\textup{cen,del}}$ for all $i\in[N]$.

\cref{lem:Xi_decdel-Xi_cendel} proves that $X_{\textup{cen,del}} \preceq X_{\textup{dec,del}}$ for all $i\in[N]$.
\begin{lem}\label{lem:Xi_decdel-Xi_cendel}
	$\Xi_{c_{\tau}}^i$ and $\Xi_{{\tau}}^i$ are the solutions of the DREs for delayed fully connected and decentralized graphs respectively. Then, $W_{\Xi}^i\defeq \Xi_{\tau}^i-\Xi_{c_{\underline{ii}}}^i \succeq 0$, where $\Xi_{c_{\underline{ii}}}^i\defeq E_{n_{\underline{i}}}^\tp\Xi_{c_{\tau}}^iE_{n_{\underline{i}}}$, and $\underline{i}$ corresponds to $\Xi_{\tau}^i$.
\end{lem}
\begin{proof}
	The DREs for $\Xi_{\tau}^i$, and $\Xi_{c_{\tau}}^i$ are subtracted to obtain the differential Lyapunov equation
	\if\MODE2
	\begin{equation*}
		\dot\Xi_{c_{\underline{ii}}}^i-\dot\Xi_{\tau}^i=(A_{\underline{ii}}+ B_{2_{\underline{i}i}} F_{\Xi}^i)^\tp W_{\Xi}^i + W_{\Xi}^i (A_{\underline{ii}}+ B_{2_{\underline{i}i}} F_{\Xi}^i)
		+(E_{m_{\underline{i}}} F_{\Xi}^i-F_{\Xi_c }^i E_{n_{\underline{i}}})^\tp D_{12}^\tp D_{12} (E_{m_{\underline{i}}} F_{\Xi}^i-F_{\Xi_c }^i E_{n_{\underline{i}}}),
	\end{equation*}
	\else
	\begin{multline*}
		\dot\Xi_{c_{\underline{ii}}}^i-\dot\Xi_{\tau}^i=(A_{\underline{ii}}+ B_{2_{\underline{i}i}} F_{\Xi}^i)^\tp W_{\Xi}^i + W_{\Xi}^i (A_{\underline{ii}}+ B_{2_{\underline{i}i}} F_{\Xi}^i)\\
		+(E_{m_{\underline{i}}} F_{\Xi}^i-F_{\Xi_c }^i E_{n_{\underline{i}}})^\tp D_{12}^\tp D_{12} (E_{m_{\underline{i}}} F_{\Xi}^i-F_{\Xi_c }^i E_{n_{\underline{i}}}),
	\end{multline*}
	\fi
	where $F_{\Xi}^i\defeq-(D_{12_{{:i}}}^\tp D_{12_{{:i}}})^{-1}	(\Xi_{\tau}^i B_{2_{{\underline{i}i}}} +C_{1_{:\underline{i}}}^\tp D_{12_{{:i}}})^\tp$, and $F_{\Xi_c}^i\defeq-(D_{12_{{:i}}}^{c^\tp} D_{12_{{:i}}}^c)^{-1}	(\Xi_{c_{\tau}}^i B_{2_{{\underline{i}i}}}^c +C_{1_{:\underline{i}}}^{c^\tp} D_{12_{{:i}}}^c)^\tp$. The rest is analogous to the proof of \cref{lem:X-X_cen}. Finally, we obtain $\Xi_{\tau}^i-\Xi_{c_{\underline{ii}}}^i-X^i+X_{{\textup{cen}}_{\underline{ii}}} \succeq 0$. Using $X^i-X_{{\textup{cen}}_{\underline{ii}}}\succeq 0$ from \cref{lem:X-X_cen}, we obtain $\Xi_{\tau}^i-\Xi_{c_{\underline{ii}}}^i \succeq 0$. 
\end{proof}


	\bibliographystyle{abbrv}
	\bibliography{dyndj}

\end{document}